\documentclass[final,preprint,3p,sort&compress]{elsarticle}
\journal{{\tt arXiv.org}}

\usepackage[dvips]{epsfig}
\usepackage{graphicx} 
\usepackage{pdfpages}
\usepackage{latexsym}
\usepackage{verbatim}
\usepackage{amsmath,amsthm}
\usepackage{amssymb}
\usepackage{subfig}
\usepackage{bbm}
\usepackage{fonts}
\usepackage{enumerate}
\usepackage{placeins}
\usepackage{standalone}
\usepackage{paralist}
 \usepackage[]{hyperref}

\newcommand{\externaltikz}[2]{\includegraphics{Externals/#1}}		

 \usepackage{relinput}
\newtheorem{theorem}{Theorem}[section]
\newtheorem{definition}[theorem]{Definition}

\newtheorem{example}[theorem]{Example}

\newtheorem{lemma}[theorem]{Lemma}

\newcounter{tikzsubfigcounter}[figure]
\renewcommand{\thetikzsubfigcounter}{\thesection.\the\numexpr\value{figure}+1\relax\alph{tikzsubfigcounter}}

\newcounter{tikzsubfigcounterinvisible}[figure]
\renewcommand{\thetikzsubfigcounterinvisible}{\thesection.\the\numexpr\value{figure}+1\relax\alph{tikzsubfigcounterinvisible}}

\newcommand{\settikzlabel}[1]{ %
\refstepcounter{tikzsubfigcounterinvisible} \label{#1} 
}

\newtheorem{thm}{\bf Theorem}[section]
\newtheorem{lem}[thm]{\bf Lemma}

\newtheorem{rem}[thm]{\bf  Remark}

\numberwithin{equation}{section}
\usepackage{fonts}

\newcommand{\changefont}[3]{
\fontfamily{#1} \fontseries{#2} \fontshape{#3} \selectfont}

\changefont{ppl}{m}{n} %Palatino

\makeatletter
\newcommand\defcase[1]{\@namedef{mycase@\the\numexpr#1\relax}}
\newcommand\picpos[1]{\@nameuse{mycase@\the\numexpr#1\relax}}
\makeatother

		\defcase{0}{}
		\defcase{11}{
						(rel axis cs:0,0)
					}
		\defcase{12}{
						(rel axis cs:0,0)
					}
		\defcase{22}{
						(rel axis cs:1.3,0)
					}

\newcommand{\secref}[1]{Section~\ref{#1}}

\newcommand{\lemref}[1]{Lemma~\ref{#1}}

\newcommand{\figref}[1]{Figure~\ref{#1}}

\newcommand{\abs}[1]{\ensuremath{\left| #1 \right|}}
\newcommand{\norm}[2]{\ensuremath{\left|\left| #1 \right|\right|_{#2}}}

\newcommand{\R}{\mathbb{R}}
\newcommand{\Rpos}{\R_{\geq 0}}
\newcommand{\scattering}{\ensuremath{\sigma_s}}
\newcommand{\absorption}{\ensuremath{\sigma_a}}

\newcommand{\source}{\ensuremath{Q}}

\newcommand{\sphere}[1][2]{\ensuremath{S^{#1}}}

\newcommand{\sign}[1]{\ensuremath{\operatorname{sign}\left(#1\right)}}

\newcommand{\distribution}[1][ ]{\ensuremath{\psi_{#1}}}
\newcommand{\distributiontzero}{\ensuremath{\distribution[\timevar=0]}}
\newcommand{\distributionboundary}{\ensuremath{\distribution[b]}}

\newcommand{\ansatz}[1][ ]{\ensuremath{\hat{\psi}_{#1}}}

\newcommand{\momentorder}{\ensuremath{N}}
\newcommand{\momentnumber}{\ensuremath{n}}
\newcommand{\basis}[1][ ]{{\ensuremath{\bb_{#1}}}} %Basis
\newcommand{\basisind}{\ensuremath{i}} %Basis
\newcommand{\basiscomp}[1][\basisind]{\ensuremath{b_{#1}}} %Basis
\newcommand{\mmbasis}[1][\momentorder]{\ensuremath{\bm_{#1}}} %Basis
\newcommand{\fmbasis}[1][\momentorder]{\ensuremath{\bff_{#1}}} %Basis
\newcommand{\pmbasis}[1][\momentorder]{\ensuremath{\bp_{#1}}} %Basis
\newcommand{\moments}[1][ ]{\ensuremath{\bu_{#1}}} %moment vector
\newcommand{\momentcomp}[1]{\ensuremath{u_{#1}}} %moment vector

\newcommand{\isotropicmoment}{\moments[\text{iso}]}
\newcommand{\normalizedmoments}[1][ ]{\ensuremath{\changefont{cmr}{m}{n}\bsphi_{#1}}} %moment vector
\newcommand{\normalizedmomentcomp}[1]{\ensuremath{\phi_{#1}}} %moment vector

\newcommand{\multipliers}[1][ ]{\ensuremath{\bsalpha_{#1}}} %moment vector
\newcommand{\eddington}{\ensuremath{\chi}} %eddington factor
\newcommand{\convexscalar}{\ensuremath{\zeta}} %scalar indicating a convex combination
\newcommand{\SC}{\ensuremath{\Omega}} %moment vector
\newcommand{\SCheight}{\ensuremath{\mu}} %moment vector
\newcommand{\SCangle}{\ensuremath{\varphi}} %moment vector
\newcommand{\LaplaceBeltrami}{\ensuremath{\Delta_\SC}} %moment vector
\newcommand{\Domain}{\ensuremath{X}} %moment vector
\newcommand{\outernormal}{\ensuremath{\bn}} %moment vector
\newcommand{\spatialVariable}{\ensuremath{\bx}} %moment vector
\newcommand{\timeint}{\ensuremath{T}} %time interval
\newcommand{\tf}{\ensuremath{t_f}} %final time
\newcommand{\timevar}{\ensuremath{t}} %final time

\newcommand{\ints}[1]{\ensuremath{\left<#1\right>}}
\newcommand{\intA}[2]{\ensuremath{\left<#1\right>_{#2}}}

\newcommand{\collisionop}{\ensuremath{\cC}}
\newcommand{\collision}[1]{\ensuremath{\collisionop\left(#1\right)}}

\newcommand{\dirac}{\ensuremath{\delta}}
\newcommand{\diracmd}{\ensuremath{\bsdelta}}

\newcommand{\indicator}[1]{\ensuremath{1_{#1}}}
\newcommand{\betafun}{\ensuremath{\mathcal{B}}}
\newcommand{\Lp}[1]{\ensuremath{L_{#1}}}
\newcommand{\RD}[2]{\ensuremath{\mathcal{R}_{#1}^{#2}}}

\newcommand{\unitnormalvector}{\ensuremath{\bsnu}}

\newcommand{\PN}[1][\momentorder]{\ensuremath{\text{P}_{#1}}}
\newcommand{\MN}[1][\momentorder]{\ensuremath{\text{M}_{#1}}}
\newcommand{\KN}[1][\momentorder]{\ensuremath{\text{K}_{#1}}}
\newcommand{\QMN}[1][\momentorder]{\ensuremath{\text{QM}_{#1}}}
\newcommand{\QKN}[1][\momentorder]{\ensuremath{\text{QK}_{#1}}}

\newcommand{\MMN}[1][\momentorder]{\ensuremath{\text{MM}_{#1}}}
\newcommand{\MKN}[1][\momentorder]{\ensuremath{\text{MK}_{#1}}}

\newcommand{\SPN}[1][\momentorder]{\ensuremath{\text{SP}_{#1}}}

\newcommand{\eigenvector}{\ensuremath{\bv}}

\newcommand{\eigenvalue}{\ensuremath{\lambda}}

\DeclareMathOperator*{\argmin}{argmin}
\newcommand{\ld}[1]{\ensuremath{{#1}_*}} %Legendre dual
\newcommand{\entropy}{\ensuremath{\eta}} %Entropy function
\newcommand{\entropyFunctional}{\ensuremath{\mathcal{H}}} %Entropy functional

\def\quand{\quad \mbox{and} \quad}

\newcommand{\x}{\ensuremath{x}}
\newcommand{\y}{\ensuremath{y}}
\newcommand{\z}{\ensuremath{z}}
\newcommand{\dx}{\partial_{\x}}
\newcommand{\dy}{\partial_{\y}}

\newcommand{\dt}{\partial_\timevar}
\newcommand{\de}{\partial}

\newcommand{\Dx}{\nabla_\spatialVariable}
\newcommand{\momentstens}[2][ ]{\ensuremath{\bu_{#1}^{\abs{#2}}}} %moment vector
\newcommand{\momentcomptens}[3][ ]{\ensuremath{u_{#1}^{\left(#2,#3\right)}}} %moment vector
\newcommand{\multiplierstens}[2][ ]{\ensuremath{\bsalpha_{#1}^{\abs{#2}}}} %moment vector
\newcommand{\multiplierscomptens}[3][ ]{\ensuremath{\alpha_{#1}^{\left(#2,#3\right)}}} %moment vector
\newcommand{\normalizedmomentstens}[2][ ]{\ensuremath{\normalizedmoments[#1]^{\abs{#2}}}} %moment vector
\newcommand{\normalizedmomentcomptens}[3][ ]{\normalizedmomentcomp{#1}^{\left(#2,#3\right)}} %moment vector
\newcommand{\spheresubset}{\ensuremath{D}}

\newcommand{\Spp}{\mathcal{S}_{++}}
\newcommand{\Smp}{\mathcal{S}_{-+}}
\newcommand{\Spm}{\mathcal{S}_{+-}}
\newcommand{\Smm}{\mathcal{S}_{--}}
\newcommand{\Sij}{\mathcal{S}_{ij}}

\newcommand{\Sxp}{\mathcal{S}_{x}^+}
\newcommand{\Sxm}{\mathcal{S}_{x}^-}
\newcommand{\Sxpm}{\mathcal{S}_{x}^\pm}
\newcommand{\Syp}{\mathcal{S}_{y}^+}
\newcommand{\Sym}{\mathcal{S}_{y}^-}
\newcommand{\Sypm}{\mathcal{S}_{y}^\pm}

\newcommand{\intsxp}[1]{\intA{#1}{\Sxp}}
\newcommand{\intsxm}[1]{\intA{#1}{\Sxm}}
\newcommand{\intsyp}[1]{\intA{#1}{\Syp}}
\newcommand{\intsym}[1]{\intA{#1}{\Sym}}

\newcommand{\intsij}[1]{\intA{#1}{\Sij}}
\newcommand{\intpp}[1]{\intA{#1}{\Spp}}
\newcommand{\intmp}[1]{\intA{#1}{\Smp}}
\newcommand{\intmm}[1]{\intA{#1}{\Smm}}
\newcommand{\intpm}[1]{\intA{#1}{\Spm}}

\newcommand{\SCx}{\ensuremath{\SC_\x}}
\newcommand{\SCy}{\ensuremath{\SC_\y}}
\newcommand{\basisindx}{\ensuremath{i_\x}}
\newcommand{\basisindy}{\ensuremath{i_\y}}

\newcommand{\absc}[1]{\ensuremath{\left| #1 \right|_c}}
\newcommand{\MMNorm}[1]{\ensuremath{\mathcal{N}\left(#1\right)}}
\def\MMconstant{\ensuremath{c}}
\newcommand{\QKalpha}[1]{\ensuremath{\alpha_{#1}}}
\newcommand{\winkel}{<\hspace{-1.25ex})\hspace{0.0ex}}

\newcommand{\QKdistcoeff}[1]{\ensuremath{c_{#1}}}
\newcommand{\QKdistdens}[1]{\ensuremath{\rho_{#1}}}
\newcommand{\QKdistdensub}[1]{\ensuremath{\overline{\rho}_{#1}}}
\newcommand{\QKdisteig}[1]{\ensuremath{\bar{\lambda}_{#1}}}
\newcommand{\QKdistind}{\ensuremath{\iota}}
\newcommand{\QKdisteigvector}{\ensuremath{\bar{\eigenvector}}}

\newcommand{\MKLBconst}[1]{\ensuremath{d_{#1}}}
\newcommand{\MKLBmatrix}{\ensuremath{S}}

\def\LinesourceWidth{\ensuremath{\sigma}}
\def\TwoBeamsWidth{\ensuremath{\sigma}}

\newcommand{\bdm}{\begin{displaymath}}
\newcommand{\edm}{\end{displaymath}}
\newcommand{\beq}{\begin{equation}}
\newcommand{\eeq}{\end{equation}}
\newcommand{\beqa}{\begin{eqnarray}}
\newcommand{\eeqa}{\end{eqnarray}}

\title{First-order quarter- and mixed-moment realizability theory and Kershaw closures for a Fokker-Planck equation in two space dimensions}
\author[fs]{Florian Schneider}
\author[jk]{Jochen Kall}
\author[ar]{Andreas Roth}
\address[fs]{Fachbereich Mathematik, TU Kaiserslautern, Erwin-Schr\"odinger-Str., 67663 Kaiserslautern, Germany, {\tt schneider@mathematik.uni-kl.de}}
\address[jk]{Fachbereich Mathematik, TU Kaiserslautern, Erwin-Schr\"odinger-Str., 67663 Kaiserslautern, Germany, {\tt kall@mathematik.uni-kl.de}}
\address[ar]{Fachbereich Mathematik, TU Kaiserslautern, Erwin-Schr\"odinger-Str., 67663 Kaiserslautern, Germany, {\tt roth@mathematik.uni-kl.de}}

\date{}

\begin{document}

\begin{abstract}
Mixed-moment models, introduced in \cite{Frank07,Schneider2014} for one space dimension, are a modification of the method of moments applied to a (linear) kinetic equation, by choosing mixtures of different partial moments. They are well-suited to handle such equations where collisions of particles are modelled with a Laplace-Beltrami operator. We generalize the concept of mixed moments to two dimension. The resulting hyperbolic system of equations has desirable properties, removing some drawbacks of the well-known $\MN[1]$ model. We furthermore provide a realizability theory for a first-order system of mixed moments by linking it to the corresponding quarter-moment theory. Additionally, we derive a type of Kershaw closures for mixed- and quarter-moment models, giving an efficient closure (compared to minimum-entropy models). The derived closures are investigated for different benchmark problems.
\end{abstract}
\maketitle

\noindent

% {\bf Key words.}

%%%%%%%%%%%%%%%%%%%%%%
\section{Introduction}
%%%%%%%%%%%%%%%%%%%%%%
The full discretization of kinetic transport equations like the Fokker-Planck equation is in general very expensive since the discretized variable resides in $\Domain\times\sphere\times[0,\tf]$ where $\Domain\subset \R^3$ and $\sphere$ denotes the unit sphere in $\R^3$. Thus, the solution of the Fokker-Planck equation is very high-dimensional. 

A common approach to reduce the dimensionality is given by the method of moments \cite{Eddington,Lev96}. One chooses a set of angular basis functions, tests the Fokker-Planck equation with it and integrates over the angular variable, removing the angular dependence while getting a system of equations in $\spatialVariable$ and $\timevar$. 
Assuming now a specific form of the underlying distribution function, different approximate systems arise. Typical examples are the well-known spherical harmonics or $\PN$ models \cite{Jea17,Eddington,Brunner2005} and their simplifications, the $\SPN$ \cite{Gel61} methods. These models are computationally inexpensive since they form an analytically closed system of hyperbolic differential equations. However, they suffer from severe drawbacks:
The $\PN$ methods are generated by closing the  balance equations with a distribution function which is a polynomial in the angular variable. This implies that this distribution function might be negative resulting in non-physical values like a negative particle density.  Additionally, in many cases a very high number of moments is needed for a reasonable approximation of the transport solution. This is in particular true in beam cases, where the exact transport solution forms a Dirac delta.
The entropy minimization $\MN$-models \cite{Min78,DubFeu99,BruHol01,Monreal2008,AllHau12} are expected to overcome this problem since their closure functions are always positive. In many situations these models perform very well. Still they produce unphysical steady-state shocks due to a zero netflux problem. 

 To improve this situation, half or partial moment models have been introduced in \cite{DubKla02,FraDubKla04}. These models work especially well in one space dimension, because they capture the potential discontinuity of the probability density in the angular variable which in 1D is well-located. If however, a Laplace-Beltrami operator is used instead of the standard integral scattering operator, i.e. scattering is extremely forward-peaked \cite{Pom92}, these half moment approximations fail \cite{Frank07,Schneider2014}.
 
 To improve this situation a new model with mixed moments was proposed in \cite{Frank07,Schneider2014} which is able to avoid this problem. Instead of choosing full or half moments, a mixture of both is used. Contrary to a typical half moment approximation, the lowest order moment (density) is kept as a full moment while all higher moments are averaged over half-spaces. This ensures the continuity of the underlying distribution function.
 
 Since in one spatial dimension mixed moments perform very well, it seems reasonable to extend them to multi-$D$. It turns out, that here a mixture of full, half and quarter moments is necessary to derive the correct mixed-moment ansatz.
 
 Realizability is the fact that a vector of moments is physically relevant, i.e. that it is the moment of a non-negative distribution function. While in 1D realizability theory for full moments \cite{Curto1991} and mixed moments \cite{Schneider2014} is completely solved, it remains an open problem in higher dimensions. We will give a realizability theory for quarter and mixed moments of order $1$ and derive, as in 1D, a corresponding Kershaw closure which provides an analytically closed system of equations, in contrast to minimum-entropy models which requires the solution of a nonlinear system of equations.

The paper is organized as follows: First, we will give a short introduction to the method of moments and its consequences for the Fokker-Planck equation. After setting up the basic notations in \secref{sec:macroscopic}, we provide necessary and sufficient conditions of order $1$ for realizability in the case of quarter moments and mixed moments in \secref{sec:Realizability}. \secref{sec:Kershaw} deals with the construction of Kershaw closures. Finally, we present some numerical tests for the derived models in \secref{sec:results} as well as conclusions and outlook in \secref{sec:conclusions}.
%%%%%%%%%%%%%%%%%%%%%%
\section{Macroscopic Models}
%%%%%%%%%%%%%%%%%%%%%%
\label{sec:macroscopic}
We consider the Fokker-Planck equation
\begin{subequations}
\label{eq:FokkerPlanckEasy}
\begin{align}
\dt\distribution + \SC\cdot\Dx\distribution + \absorption\distribution = \frac{\scattering}{2}\LaplaceBeltrami \distribution + \source
\end{align}
which describes the densities of particles with speed $\SC\in\sphere$ at position $\spatialVariable\in\Domain\subseteq\R^3$ and time $\timevar$ under the events of scattering (proportional to $\scattering\left(\timevar,\spatialVariable\right)$), absorption (proportional to $\absorption\left(\timevar,\spatialVariable\right)$) and emission (proportional to $\source\left(\timevar,\spatialVariable,\SC\right)$). The equation is supplemented with initial condition and Dirichlet boundary conditions:
\begin{align}
\distribution(0,\spatialVariable,\SC) &= \distributiontzero(\spatialVariable,\SC) &\text{for } \spatialVariable\in\Domain, \SC\in\sphere\\
\distribution(\timevar,\spatialVariable,\SC) &= \distributionboundary(\timevar,\spatialVariable,\SC) &\text{for } \timevar\in\timeint, \spatialVariable\in\partial\Domain, \outernormal\cdot\SC<0
\end{align}

where $\outernormal$ is the outward unit normal vector in $\spatialVariable\in\partial\Domain$.
\end{subequations}

Similar to \cite{Seibold2012} we assume that geometry, initial and boundary conditions are independent of the $\z$-direction, resulting in a solution $\distribution$ which is also $\z$-independent. Therefore \eqref{eq:FokkerPlanckEasy} can be reduced to $\spatialVariable\in\Domain\subseteq\R^2$. Parametrizing $\SC$ in spherical coordinates and taking the symmetry reduction into account we obtain
\begin{align}
\label{eq:SphericalCoordinates}
\SC = \left(\sqrt{1-\SCheight^2}\cos(\SCangle),\sqrt{1-\SCheight^2}\sin(\SCangle)\right)^T =: \left(\SCx,\SCy\right)^T
\end{align}
where $\SCangle\in[0,2\pi]$ is the azimuthal and $\SCheight\in[-1,1]$ the cosine of the polar angle. Then the Laplace-Beltrami operator on the unit sphere can be written as 
\begin{align}
\label{eq:LaplaceBeltrami}
\LaplaceBeltrami \distribution = \cfrac{d}{d\SCheight}\left(\left(1-\SCheight^2\right)\cfrac{d\distribution}{d\SCheight}\right) + \cfrac{1}{1-\SCheight^2}\cfrac{d^2\distribution}{d\SCangle^2}
\end{align}

\begin{definition}
The vector of functions $\basis:\sphere\to\R^{\momentnumber}$ consisting of $\momentnumber$ basis functions $\basiscomp[i]$, $\basisind=0,\ldots\momentnumber-1$ of maximal \emph{order} $\momentorder$ (in $\SC$) is called an \emph{angular basis}.

The so-called \emph{moments} $\moments=\left(\momentcomp{0},\ldots,\momentcomp{\momentnumber-1}\right)^T$ of a given distribution function $\distribution$ are then defined by
\begin{align}
\label{eq:moments}
\moments = \int\limits_{\sphere} {\basis}\distribution~d\SC =: \ints{\basis\distribution}
\end{align}
where the integration is performed component-wise.
\end{definition}

Equations for $\moments$ can then be obtained by multiplying \eqref{eq:FokkerPlanckEasy} with $\basis$ and integration over $\sphere$: 
\begin{align*}
\ints{\basis\dt\distribution}+\ints{\basis\Dx\cdot\SC\distribution} + \ints{\basis\absorption\distribution} = \scattering\ints{\basis\collision{\distribution}}+\ints{\basis\source}
\end{align*}
Collecting known terms, and interchanging integration and differentiation where possible, the moment system has the form
\begin{align}
\label{eq:MomentSystemUnclosed}
\dt\moments+\dx\ints{\SCx \basis\distribution}+\dy\ints{\SCy \basis\distribution}  + \absorption\moments = \cfrac{\scattering}{2}\ints{\basis\LaplaceBeltrami\distribution}+\ints{\basis\source}
\end{align}

Depending on the choice of $\basis$ the terms $\ints{\SCx \basis\distribution}$, $\ints{\SCy \basis\distribution}$ and in some cases even $\ints{\basis\collision{\distribution}}$ cannot be given explicitly in terms of $\moments$. Therefore an ansatz $\ansatz$ has to be made for $\distribution$ closing the unknown terms. This is called the \emph{moment-closure problem}.

In this paper the ansatz density $\ansatz$ is reconstructed from the moments $\moments$ by minimizing the entropy-functional 
 \begin{align}
 \label{eq:entropyFunctional}
 \entropyFunctional(\distribution) = \ints{\entropy(\distribution)}
 \end{align}
 under the moment constraints
 \begin{align}
 \label{eq:MomentConstraints}
 \ints{\basis\distribution} = \moments.
 \end{align}
The kinetic entropy density $\entropy:\R\to\R$ is strictly convex and twice continuously differentiable and the minimum is simply taken over all functions $\distribution = \distribution(\SC)$ such that 
  $\entropyFunctional(\distribution)$ is well defined. The obtained ansatz $\ansatz = \ansatz[\moments]$, solving this constrained optimization problem, is given by
 \begin{equation}
  \ansatz[\moments] = \argmin\limits_{\distribution:\entropy(\distribution)\in\Lp{1}}\left\{\ints{\entropy(\distribution)}
  : \ints{\basis \distribution} = \moments \right\}.
 \label{eq:primal}
 \end{equation}
This problem, which must be solved over the space-time mesh, is typically solved through its strictly convex finite-dimensional dual,
 \begin{equation}
  \multipliers(\moments) := \argmin_{\tilde{\multipliers} \in \R^{\momentnumber}} \ints{\ld{\entropy}(\basis^T 
   \tilde{\multipliers})} - \moments^T \tilde{\multipliers},
 \label{eq:dual}
 \end{equation}
where $\ld{\entropy}$ is the Legendre dual of $\entropy$. The first-order necessary conditions for the multipliers $\multipliers(\moments)$ show that the solution to \eqref{eq:primal} has the form
 \begin{equation}
  \ansatz[\moments] = \ld{\entropy}' \left(\basis^T \multipliers(\moments) \right)
 \label{eq:psiME}
 \end{equation}
where $\ld{\entropy}'$ is the derivative of $\ld{\entropy}$.\\

This approach is called the \emph{minimum-entropy closure} \cite{Levermore1996}.

The kinetic entropy density $\entropy$ can be chosen according to the 
physics being modelled.
As in \cite{Levermore1996,Hauck2010}, Maxwell-Boltzmann entropy%
 \begin{align}
 \label{eq:EntropyM}
  \entropy(\distribution) = \distribution \log(\distribution) - \distribution
 \end{align}
is used, thus $\ld{\entropy}(p) = \ld{\entropy}'(p) = \exp(p)$. This entropy is used for non-interacting particles as in an ideal gas or an ensemble of photons.

A closed system of equations for $\moments$ remains after substituting $\distribution$ in \eqref{eq:MomentSystemUnclosed} with $\ansatz[\moments]$:
\begin{align}
\label{eq:MomentSystemClosed}
\dt\moments+\dx\ints{\SCx \basis\ansatz[\moments]}+\dy\ints{\SCy \basis\ansatz[\moments]}  + \absorption\moments = \cfrac{\scattering}{2}\ints{\basis\LaplaceBeltrami\ansatz[\moments]}+\ints{\basis\source}
\end{align}

Also note that using the entropy $\entropy(\distribution) = \frac12\distribution^2$ the linear ansatz 
\begin{align}
\label{eq:PnAnsatz}
\ansatz[\moments] = \basis^T \multipliers(\moments)
\end{align}
remains. If the angular basis is chosen as spherical harmonics of order $\momentorder$, \eqref{eq:MomentSystemClosed} turns into the classical $\PN$ model \cite{Blanco1997,Brunner2005,Seibold2012}.
\subsection{Angular bases}
This moment approach strongly depends on the choice of the ansatz $\ansatz$ and the angular basis $\basis$. In the following sections we will shortly derive the different angular bases for the models presented here. These bases will be generally collected in the basis-vector $\basis(\SC)$. If we need to further distinguish between the models, the corresponding symbols defined in the following sections will be used. If a result is independent on the choice of the basis we will just use $\basis$ as symbol.

\subsubsection{Full moments}
The full-moment basis $\fmbasis$ of order $\momentorder$ consists of the tensorial powers of $\SC$, i.e. (by abusing notation)
\begin{align}
\fmbasis = \left(1,\SC,\SC\otimes\SC,\SC^{\otimes 3},\ldots,\SC^{\otimes\momentorder}\right)^T.
\end{align}

The corresponding tensorial moment of order $\basisind$ is given by 
\begin{align}
\momentstens{\basisind} = \ints{\SC^{\otimes \basisind}\distribution},
\end{align}
which consists of the scalar moments
\begin{align}
\momentcomptens{\basisindx}{\basisindy} = \ints{\SCx^{\basisindx}\SCy^{\basisindy}\distribution},\quad\quad \basisindx,\basisindy\geq 0, \basisindx+\basisindy=\basisind.
\end{align}

Note that an equivalent system can be obtained by using the corresponding real-valued spherical harmonics of order $\momentorder$ as angular basis \cite{Blanco1997,Brunner2005,Seibold2012}.

\subsubsection{Quarter moments}
We follow the approach in \cite{Frank2006} where general partial as well as the special case of quarter moments in two space-dimensions are treated. The main idea is not to integrate over the whole sphere $\sphere$ but over subsets of it.

\begin{definition}
For $\spheresubset\subseteq \sphere$ and $\distribution\in \Lp{2}(\spheresubset,\R)$ we define its tensorial moment by
\begin{align}
\momentstens[\spheresubset]{\basisind} = \ints{\indicator{\spheresubset}\SC^{\otimes\basisind}\distribution} =: \intA{\SC^{\otimes\basisind}\distribution}{\spheresubset}.
\end{align}
As for full moments the corresponding components of the tensorial moments are given by 
\begin{align}
\momentcomptens[\spheresubset]{\basisindx}{\basisindy} = \intA{\SCx^{\basisindx}\SCy^{\basisindy}\distribution}{\spheresubset},\quad\quad \basisindx,\basisindy\geq 0, \basisindx+\basisindy=\basisind.
\end{align}

In this paper, $\spheresubset$ will be one of the following quarterspaces
 \begin{align*}
\Smp &= \left\{\SC~|~\SCheight\in[-1,1], \SCangle\in[\frac{\pi}{2},\pi]\right\} ,~&\Spp &= \left\{\SC~|~\SCheight\in[-1,1], \SCangle\in[0,\frac{\pi}{2}]\right\},\\
\Smm &= \left\{\SC~|~\SCheight\in[-1,1], \SCangle\in[\pi,\frac{3\pi}{2}]\right\},~&\Spm &= \left\{\SC~|~\SCheight\in[-1,1], \SCangle\in[\frac{3\pi}{2},2\pi]\right\},\\
\intertext{or halfspaces}
\Sxp &= \left\{\SC~|~\SCheight\in[-1,1], \SCangle\in[-\frac{\pi}{2},\frac{\pi}{2}]\right\},&\Sxm &= \left\{\SC~|~\SCheight\in[-1,1], \SCangle\in[\frac{\pi}{2},\frac{3\pi}{2}]\right\},\\
\Syp &= \left\{\SC~|~\SCheight\in[-1,1], \SCangle\in[0,\pi]\right\},&\Sym &= \left\{\SC~|~\SCheight\in[-1,1], \SCangle\in[\pi, 2\pi]\right\}.
 \end{align*}
 Note that for $\spheresubset = \sphere$ we have that $\momentstens[\spheresubset]{\basisind}=\momentstens{\basisind}$.\\
\end{definition}

For pure quarter moments, we will have $\spheresubset=\Sij$ for $i,j\in\{+,-\}$. The corresponding basis for quadrant $ij$ is then given by
\begin{align*}
\pmbasis[\Sij] = \indicator{\Sij}\cdot(1,\SC,\SC\otimes\SC,\SC^{\otimes 3},\ldots,\SC^{\otimes\momentorder})^T
\end{align*}
where $\cdot$ should be understood as multiplication with every component. Consequently the complete set of basis functions is $\pmbasis = \left(\pmbasis[\Spp],\pmbasis[\Smp],\pmbasis[\Smm],\pmbasis[\Spm]\right)$.

\subsubsection{Mixed moments}
As will be shown below the quarter-moment basis exhibits undesired properties when applied to the Laplace-Beltrami operator. This has inspired the works in \cite{Frank07,Schneider2014} where so-called mixed moments were developed. The main problem of the quarter-moment basis is that the ansatz $\ansatz$ is not continuous in $\SC$ which is necessary for the solution of \eqref{eq:FokkerPlanckEasy} in one space-dimension. There, mixed moments where constructed by starting from a half-moment ansatz and demanding continuity of the ansatz \eqref{eq:psiME} with respect to this basis. There, it suffices to choose a full zeroth-order moment and half-moments for all higher order moments \cite{Schneider2014}.

The construction of mixed moments in two dimensions works in the same spirit. We start with the general quarter-moment basis $\pmbasis$ and demand continuity of the ansatz $\ansatz[\moments]$.

Having \eqref{eq:psiME} in mind, we obtain multipliers $\multipliers[ij]$ for every quadrant $\Sij$. For example, for $\momentorder = 1$ we have 

\begin{align*}
\ansatz[\moments]\left|_{\Sij}\right. = \ld{\entropy}'\left(\multiplierscomptens[\Sij]{0}{0}+\multiplierstens[\Sij]{1}\cdot \SC\right).
\end{align*}
At the poles of the sphere ($\SCheight = \pm 1$) we have $\SC = \left(0,0\right)^T$, which implies that $\ansatz[\moments]$ is continuous only if $\multiplierscomptens[\Sij]{0}{0} = \multiplierscomptens{0}{0}$ for all $i,j\in\{+,-\}$. Similarly it holds that along the quarter-space boundaries (i.e. $\SCangle = k\frac{\pi}{2}$, $k =0,\ldots,3$) some of the multipliers have to be the same (exactly those whose component of $\SC$ does not vanish on the boundary, e.g. $\multiplierscomptens[\Spp]{1}{0} = \multiplierscomptens[\Spm]{1}{0}$ since at $\SCangle = 0$ we have $\SCx = 1\neq 0$).

Accordingly we obtain the moments
\begin{align}
\momentcomptens{0}{0} &= \ints{\distribution} & \momentcomptens[\spheresubset]{\basisindx}{\basisindy} &= \intA{\SCx\SCy\distribution}{\spheresubset}\\
\momentcomptens[\Sxp]{\basisind}{0} &=\intsxp{\SCx\distribution} & \momentcomptens[\Sxm]{\basisind}{0} &=\intsxm{\SCx\distribution}\\
\momentcomptens[\Syp]{0}{\basisind} &=\intsyp{\SCy\distribution} & \momentcomptens[\Sym]{0}{\basisind} &=\intsym{\SCy\distribution}
\end{align}
for  $\basisind = 1,\ldots,\momentorder$, $\basisindx+\basisindy= \basisind$, $\basisindx,\basisindy\geq 0$, and $\spheresubset\in\{\Spp,\Smp,\Smm,\Spm\}$.

In contrast to the one-dimensional setting one full moment, half moments for the basis functions contributing to either $\x$ or $\y$ direction, and quarter moments for the basis functions which contribute to both directions occur.

Note that in the fully three-dimensional setting the decomposition in $\z$-direction has to be taken into account, leading to octants instead of quadrants.

To embed this in the framework we choose our basis function $\mmbasis$ of order $\momentorder$ as
\begin{align}
\nonumber\mmbasis = (1,&\SCx\indicator{\Sxp},\ldots,\SCx^\momentorder\indicator{\Sxp},\SCx\indicator{\Sxm},\ldots,\SCx^\momentorder\indicator{\Sxm},\\
\nonumber&\SCy\indicator{\Syp},\ldots,\SCy^\momentorder\indicator{\Syp},\SCy\indicator{\Sym},\ldots,\SCy^\momentorder\indicator{\Sym},\\
\nonumber&\SCx\SCy\indicator{\Spp},\SCx^2\SCy\indicator{\Spp},\SCx\SCy^2\indicator{\Spp},\ldots,\SCx\SCy^{\momentorder-1}\indicator{\Spp})\\
&\hskip 4cm\vdots\\
\nonumber&\SCx\SCy\indicator{\Spm},\SCx^2\SCy\indicator{\Spm},\SCx\SCy^2\indicator{\Spm},\ldots,\SCx\SCy^{\momentorder-1}\indicator{\Spm})^T\\
\nonumber=: (\mmbasis[\sphere],&\mmbasis[\Sxp],\mmbasis[\Sxm],\mmbasis[\Syp],\mmbasis[\Sym],\mmbasis[\Spp],\mmbasis[\Smp],\mmbasis[\Smm],\mmbasis[\Spm])^T.
\end{align}

Formulas to efficiently calculate the appearing integrals in case of a linear ansatz can be found in Appendix \ref{sec:CalculationPMintegrals}.

\subsection{Moments of the Laplace-Beltrami operator}
All that remains to obtain a closed set of equations in \eqref{eq:MomentSystemClosed} is to correctly evaluate $\ints{\basis\LaplaceBeltrami\ansatz[\moments]}$. This can be done using the formal self-adjointness of the Laplace-Beltrami operator, using
\begin{align*}
\ints{\basis \LaplaceBeltrami \distribution} = \ints{\distribution\LaplaceBeltrami\basis }.
\end{align*} 
Due to this, the calculation of these integrals does not a priori depend on the choice of the ansatz but is true for every $\distribution$.

\subsubsection{Full moments $\fmbasis$-basis}
Straight-forward calculations show that for $\basisindx,\basisindy\geq 0$, $\basisindx+\basisindy = \basisind$ it holds that 
\begin{align}
\LaplaceBeltrami \SCx^{\basisindx}\SCy^{\basisindy} = -\basisind\left(\basisind+1\right)\SCx^{\basisindx}\SCy^{\basisindy}+\basisindx\left(\basisindx-1\right)\SCx^{\basisindx-2}\SCy^{\basisindy}+\basisindy\left(\basisindy-1\right)\SCx^{\basisindx}\SCy^{\basisindy-2}.
\end{align}
Consequently, the corresponding moments are given by
\begin{align*}
\ints{\SCx^{\basisindx}\SCy^{\basisindy}\LaplaceBeltrami\distribution} = -\basisind\left(\basisind+1\right)\momentcomptens{\basisindx}{\basisindy}+\basisindx\left(\basisindx-1\right)\momentcomptens{\basisindx-2}{\basisindy}+\basisindy\left(\basisindy-1\right)\momentcomptens{\basisindx}{\basisindy-2}.
\end{align*}

\subsubsection{Quarter moments $\pmbasis$-basis}
We observe the following relations on the quadrant $\Spp$ for $\basisindx,\basisindy> 0$, $\basisindx+\basisindy= \basisind$:
\begin{align}
 \label{eq:QMLaplaceBeltramizerothmoment}
\LaplaceBeltrami\indicator{\Spp} &= \cfrac{\dirac^\prime\!\left(\SCangle\right)+\dirac^\prime\!\left(\frac{\pi}{2}-\SCangle\right)}{1-\SCheight^2}\\
\LaplaceBeltrami\indicator{\Spp}\SCx^{\basisindx} &= \sqrt{1-\SCheight^2}^{\basisindx-2}\dirac^\prime\!\left(\SCangle\right) + \basisindx(\basisindx-1)\SCx^{\basisindx-2}\indicator{\Spp} -\basisindx(\basisindx+1)\SCx^{\basisindx}\indicator{\Spp} \\
\nonumber
\LaplaceBeltrami\indicator{\Spp}\SCx^{\basisindy} &= \sqrt{1-\SCheight^2}^{\basisindy-2}\dirac^\prime\!\left(\frac{\pi}{2}-\SCangle\right) + \basisindy(\basisindy-1)\SCy^{\basisindy-2}\indicator{\Spp} -\basisindy(\basisindy+1)\SCy^{\basisindy}\indicator{\Spp} \\
\nonumber
\LaplaceBeltrami \indicator{\Spp}\SCx^{\basisindx}\SCy^{\basisindy} &= \indicator{\Spp}\left(-\basisind\left(\basisind+1\right)\SCx^{\basisindx}\SCy^{\basisindy}+\basisindx\left(\basisindx-1\right)\SCx^{\basisindx-2}\SCy^{\basisindy}+\basisindy\left(\basisindy-1\right)\SCx^{\basisindx}\SCy^{\basisindy-2}\right)\\
\nonumber
&+\cfrac{\basisindy\SCx^{\basisindx+1}\SCy^{\basisindy-1}+\basisindx\SCx^{\basisindx-1}\SCy^{\basisindy+1}}{1-\SCheight^2}\left(\dirac\left(\SCangle-\cfrac{\pi}{2}\right)-\dirac\left(\SCangle\right)\right).
\end{align}
Similarly, the corresponding quantities in the other quarter-spaces can be obtained.

The moments of $\LaplaceBeltrami\distribution$ include the evaluation of the microscopic values $\partial_\SCangle \distribution$ and $\distribution$ at the quarter-sphere boundaries. 
Note that, similar to the one-dimensional case, the mass-conservation property of the Laplace-Beltrami operator (i.e. $\ints{\LaplaceBeltrami\distribution}=0$) is usually violated. This can be easily seen by summing up \eqref{eq:QMLaplaceBeltramizerothmoment} over all quadrants, observing that $\de_\SCangle\distribution$ at the angles $\frac{k\pi}{2}$, $k=0,\ldots,3$ remains. Also note that this quantity is not rigorously defined for minimum-entropy closures due to the discontinuity in the ansatz $\ansatz$.

\subsubsection{Mixed moments $\mmbasis$-basis}
\label{sec:MixedMomentsLB}
We observe for $\basisindx,\basisindy> 0$ the following:
\begin{align*}
\LaplaceBeltrami 1 &= 0\\
\LaplaceBeltrami \indicator{\spheresubset}\SCx^{\basisindx} &= -\basisindx\left(\basisindx+1\right) \indicator{\spheresubset}\SCx^{\basisindx} +\basisindx\left(\basisindx-1\right) \indicator{\spheresubset}\SCx^{\basisindx-2}\\&~\quad\pm \indicator{\basisindx=1}\frac{\dirac\!\left(\SCangle - \frac{\pi}{2}\right) + \dirac\!\left(\SCangle-\frac{3\pi}{2}\right)}{\sqrt{1 - {\mu}^2}} &\text{ for } \spheresubset\in\{\Sxp,\Sxm\}\\
\LaplaceBeltrami \indicator{\spheresubset}\SCy^{\basisindy} &= -\basisindy\left(\basisindy+1\right) \indicator{\spheresubset}\SCy^{\basisindy} +\basisindy\left(\basisindy-1\right) \indicator{\spheresubset}\SCy^{\basisindy-2}\\&~\quad\pm \indicator{\basisindy=1}\frac{\dirac\!\left(\SCangle\right) + \dirac\!\left(\SCangle-\pi\right)}{\sqrt{1 - {\mu}^2}} &\text{ for } \spheresubset\in\{\Syp,\Sym\}.\\
\end{align*}
The calculations for $\LaplaceBeltrami \indicator{\Sij}\SCx^{\basisindx}\SCy^{\basisindy}$ are equivalent to those of the quarter-moment basis.

Note that all these calculations are closure-independent. We therefore need to calculate $\momentcomptens[\Sxpm]{0}{0}$, $\momentcomptens[\Sypm]{0}{0}$, $\momentcomptens[\Sij]{0}{0}$, $\momentcomptens[\Sij]{\momentorder}{0}$ ,$\momentcomptens[\Sij]{0}{\momentorder}$ and the semi-microscopic quantities $\int_{-1}^1\ansatz\left(\SCheight,k\frac{\pi}{2}\right)~d\SCheight$ for $k=0,\ldots 3$.

\section{Realizability}
\label{sec:Realizability}
In this section we will define precisely the concept of realizability. Furthermore we give necessary and sufficient conditions for first order quarter and mixed moment models.
\begin{definition}[Realizability]
A moment vector $\moments$ is said to be \emph{realizable} with respect to basis $\basis$ if there exists a non-negative distribution $\distribution(\SC)\geq 0$ such that $\moments = \ints{\basis\distribution}$. The \emph{realizable set} is defined as
\begin{align}
\label{eq:RealizableSet}
\RD{\basis}{} = \{\moments\in\R^\momentnumber~|~ \exists\distribution\geq 0 \text{ s.t. } \moments = \ints{\basis\distribution}\}
\end{align}
$\distribution$ is then called a \emph{realizing distribution}.
\end{definition}

Note that \eqref{eq:dual} is solvable if and only if $\moments\in\RD{\basis}{}$.\\

A standard example for realizability are the realizability conditions of first order in the full-moment setting.

\begin{example}
Let $\basis = \fmbasis[1] = \left(1,\Omega\right)^T$ and $\moments = \left(\momentcomptens{0}{0},\momentstens{1}\right)^T$. Then $\moments\in\RD{\fmbasis[1]}{}$ if and only if \cite{Ker76}
\begin{align}
\norm{\momentstens{1}}{2} \leq \momentcomptens{0}{0}.
\end{align}

\end{example}

\subsection{Quarter moments}
In this section first order necessary and sufficient conditions for realizability of a quarter-moment vector will be given. This will be used later to derive the corresponding conditions for mixed moments.
\begin{lemma}
\label{lem:QMFORealizability}
For a vector of moments $\moments = \left(\momentcomptens[\Sij]{0}{0},\momentstens[\Sij]{1}\right)^T\in\R^3$ it is necessary and sufficient for the existence of a non-negative measure $\distribution$ which realizes $\moments$ with respect to $\pmbasis[\Sij]$ that
\begin{align}
\label{eq:QuarterMomentsRealizable}
\norm{\momentstens[\Sij]{1}}{2}\leq \momentcomptens[\Sij]{0}{0}
\end{align}
and the normalized first moment 
\begin{align}
\normalizedmomentstens[\Sij]{1} := \frac{\momentstens[\Sij]{1}}{\momentcomptens[\Sij]{0}{0}}\label{eq:NormalizedMomentsQM1}
\end{align}
satisfies $\normalizedmomentstens[\Sij]{1} \in\Sij$.
\end{lemma}
\begin{proof}
We only prove the statement for $\Spp$. The proof for the other quadrants works similarly.\\
Assume that $\distribution\geq 0$ in $\Spp$. Since $\norm{\SC}{2}\leq 1$ we obtain
\begin{align*}
\norm{\momentstens[\Spp]{1}}{2} = \norm{\intpp{\SC\distribution}}{2} \leq \intpp{\norm{\SC}{2}\distribution}\leq \intpp{\distribution} = \momentcomptens[\Spp]{0}{0}
\end{align*}
showing the necessity of \eqref{eq:QuarterMomentsRealizable}. Since $\SC\in\Spp$ we obtain $\SCx,\SCy\geq 0$ implying that $\momentcomptens[\Spp]{1}{0},\momentcomptens[\Spp]{0}{1}\geq 0$. Together with \eqref{eq:QuarterMomentsRealizable} we have $\normalizedmomentstens[\Spp]{1} \in\Spp$.

To show the sufficiency of \eqref{eq:QuarterMomentsRealizable} we give a realizing distribution function. A possible (but not necessarily unique) candidate is given by
\begin{align}
\label{eq:psiQM}
\distribution[\Spp] = \momentcomptens[\Spp]{0}{0}\diracmd\left(\SC-\normalizedmomentstens[\Spp]{1}\right)
\end{align} 
where $\diracmd$ denotes the multi-dimensional Dirac-delta distribution\footnote{We assume for notational simplicity that $\diracmd$ has mass $1$ even on the boundary of integration.}. If $\normalizedmomentstens[\Spp]{1}\in\Spp$ the distribution function is supported in $\Spp$. Thus
\begin{align*}
\ints{\distribution[\Spp]} = \momentcomptens[\Spp]{0}{0}\quand \ints{\SC\distribution[\Spp]} = \momentcomptens[\Spp]{0}{0}\normalizedmomentstens[\Spp]{1} \stackrel{\eqref{eq:NormalizedMomentsQM1}}{=} \momentstens[\Spp]{1}.
\end{align*}
Therefore, $\distribution[\Spp]$ is a realizing distribution for $\moments$ under the given assumptions.
\end{proof}
\subsection{Mixed moments}
With the knowledge of \lemref{lem:QMFORealizability} we are able to provide the realizability conditions for mixed moments of order $1$. ‚\\
\begin{theorem}[First order necessary and sufficient conditions]
\label{thm:NecSuffFOMM}
For a vector of moments $$\moments = (\momentcomptens{0}{0},\momentcomptens[\Sxp]{1}{0},\momentcomptens[\Sxm]{1}{0},\momentcomptens[\Syp]{0}{1},\momentcomptens[\Sym]{0}{1})^T\in\R^5$$ it is necessary and sufficient for the existence of a non-negative measure $\distribution$ which realizes $\moments$ with respect to $\mmbasis[1]$ that
\begin{align}
\label{eq:NecSuffFOMM}
\sqrt{\left(\momentcomptens[\Sxp]{1}{0}-\momentcomptens[\Sxm]{1}{0}\right)^2 + \left(\momentcomptens[\Syp]{0}{1}-\momentcomptens[\Sym]{0}{1}\right)^2}\leq \momentcomptens{0}{0}
\end{align}
and $\momentcomptens[\Sxp]{1}{0},\momentcomptens[\Syp]{0}{1},-\momentcomptens[\Sxm]{1}{0},-\momentcomptens[\Sym]{0}{1}\geq 0$.
\end{theorem}
\begin{proof}
We start again with the necessity of \eqref{eq:NecSuffFOMM}:\\
Note that the vector of component-wise absolute values $\absc{\SC} := (\abs{\SCx},\abs{\SCy})$ satisfies $\norm{\absc{\SC}}{2} = \norm{\SC}{2} \leq 1$. Therefore $1-\unitnormalvector^T\absc{\SC}\geq 0$ for every unitvector $\unitnormalvector\in\sphere[1]$, implying
\begin{align*}
0\leq \ints{\left(1-\unitnormalvector^T\absc{\SC}\right)\distribution}.
\end{align*}
This can be reformulated to
\begin{align}
\label{eq:nuMMFirstOrder}
\unitnormalvector^T\widetilde{\moments}:=\unitnormalvector^T\begin{pmatrix}
\momentcomptens[\Sxp]{1}{0}-\momentcomptens[\Sxm]{1}{0}\\\momentcomptens[\Syp]{0}{1}-\momentcomptens[\Sym]{0}{1}
\end{pmatrix}\leq \momentcomptens{0}{0}
\end{align}
The left hand side will be extremal if $\unitnormalvector$ is collinear to $\widetilde{\moments}$, i.e. $\unitnormalvector = \frac{\widetilde{\moments}}{\norm{\widetilde{\moments}}{2}}$.
Thus \eqref{eq:nuMMFirstOrder} implies \eqref{eq:NecSuffFOMM}. The sign-constraints on the moments follow again from the signs of $\SCx$ and $\SCy$ in the corresponding halfspaces.
 
For sufficiency we give again a reproducing distribution. In the following we will make use that under \eqref{eq:NecSuffFOMM} the reduced moment vector $\widetilde{\moments}$ is located in $\momentcomptens{0}{0}\Spp$. Similarly the quantities
\begin{align}
\label{eq:phiij}
\underline{\normalizedmomentstens[\Sij]{1}} :=   \left(i\frac{\momentcomptens[\Sxp]{1}{0}-\momentcomptens[\Sxm]{1}{0}}{\momentcomptens{0}{0}},j\frac{\momentcomptens[\Syp]{0}{1}-\momentcomptens[\Sym]{0}{1}}{\momentcomptens{0}{0}}\right),\quad i,j\in\{+,-\}
\end{align}
will be in $\Sij$. Reformulating \eqref{eq:NecSuffFOMM} in terms of the normalized moments collected in $\normalizedmomentstens{1}$ we see that 
\begin{align}
\label{eq:NecSuffFOMMNorm}
\MMNorm{\normalizedmomentstens{1}} := \sqrt{\left(\normalizedmomentcomptens[\Sxp]{1}{0}-\normalizedmomentcomptens[\Sxm]{1}{0}\right)^2 + \left(\normalizedmomentcomptens[\Syp]{0}{1}-\normalizedmomentcomptens[\Sym]{0}{1}\right)^2}\leq 1
\end{align}
Also note that $\MMNorm{\normalizedmomentstens{1}}  = \norm{\underline{\normalizedmomentstens[\Sij]{1}} }{2}$. We now embed the moments $\underline{\normalizedmomentstens[\Sij]{1}}\in\R^2$ into the normalized mixed-moment space (which is a subset of $\R^4$) in the following way. Define
\begin{subequations}
\label{eq:momvectij}
\begin{align}
\normalizedmomentstens[\Spp]{1} &:= \left(\normalizedmomentcomptens[\Sxp]{1}{0}-\normalizedmomentcomptens[\Sxm]{1}{0},0,\normalizedmomentcomptens[\Syp]{0}{1}-\normalizedmomentcomptens[\Sym]{0}{1},0\right)^T\\
\normalizedmomentstens[\Smp]{1} &:= \left(0,\normalizedmomentcomptens[\Sxm]{1}{0}-\normalizedmomentcomptens[\Sxp]{1}{0},\normalizedmomentcomptens[\Syp]{0}{1}-\normalizedmomentcomptens[\Sym]{0}{1},0\right)^T\\
\normalizedmomentstens[\Smm]{1} &:= \left(0,\normalizedmomentcomptens[\Sxm]{1}{0}-\normalizedmomentcomptens[\Sxp]{1}{0},0,\normalizedmomentcomptens[\Sym]{0}{1}-\normalizedmomentcomptens[\Syp]{0}{1}\right)^T\\
\normalizedmomentstens[\Spm]{1} &:= \left(\normalizedmomentcomptens[\Sxp]{1}{0}-\normalizedmomentcomptens[\Sxm]{1}{0},0,\normalizedmomentcomptens[\Sym]{0}{1}-\normalizedmomentcomptens[\Syp]{0}{1}\right)^T,
\end{align}
which then fulfils $\MMNorm{\normalizedmomentstens[\Spp]{1}} = \ldots = \MMNorm{\normalizedmomentstens[\Spm]{1}} =: \MMconstant$.
\end{subequations}
Furthermore, any convex combination of those modified moment vectors along neighbouring quadrants (e.g. $\Spp$ and $\Smp$) lies on an isoline of $\MMNorm{\cdot}$, i.e. for $\convexscalar\in[0,1]$ we have
\begin{align*}
\MMNorm{\convexscalar\normalizedmomentstens[\Spp]{1}+(1-\convexscalar)\normalizedmomentstens[\Smp]{1}} = \MMconstant = \MMNorm{\convexscalar\normalizedmomentstens[\Spm]{1}+(1-\convexscalar)\normalizedmomentstens[\Smm]{1}}
\end{align*}
and analogously for the other half-space combinations. This is visualized in \figref{fig:ThreeMomBoundaryFO}.
\begin{figure}
\centering
\externaltikz{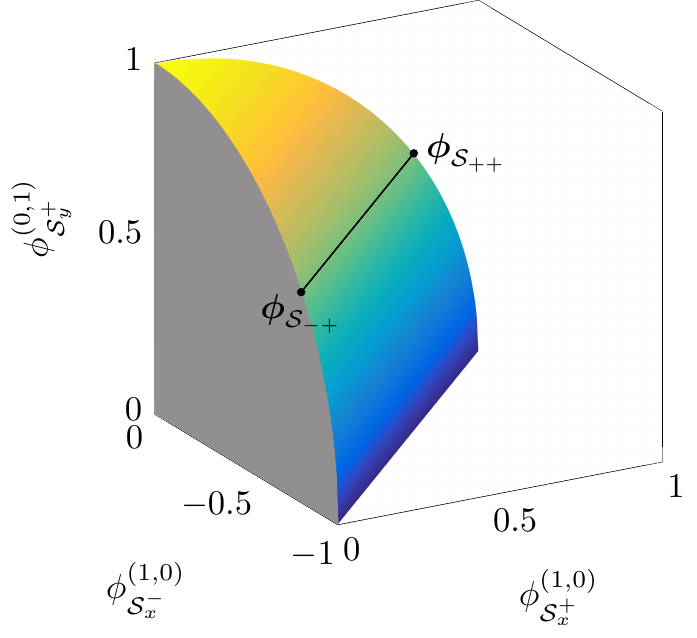}{\relinput{Images/Realizability2/Realizability2.tex}}
\caption{Linear interpolation between two realizability boundaries in the projected three-space $(\normalizedmomentcomptens[\Sxp]{1}{0},\normalizedmomentcomptens[\Sxm]{1}{0},\normalizedmomentcomptens[\Syp]{0}{1})$ along an isoline of $\MMNorm{\normalizedmomentstens{1}}$. The realizable set with respect to the quadrant $\Smp$ is plotted in grey.}
\label{fig:ThreeMomBoundaryFO}
\end{figure}
It can be shown that $\normalizedmomentstens{1}$ (and consequently $\moments$) can be written as a convex combination of the moments $\normalizedmomentstens[\Spp]{1}$ to $\normalizedmomentstens[\Spm]{1}$. Indeed, defining
\begin{gather}
\label{eq:MM1Interpolation}
\convexscalar_1:=\cfrac{\normalizedmomentcomptens[\Syp]{0}{1}}{\normalizedmomentcomptens[\Syp]{0}{1}-\normalizedmomentcomptens[\Sym]{0}{1}}\in[0,1]\quand  \convexscalar_2:=\cfrac{\normalizedmomentcomptens[\Sxp]{1}{0}}{\normalizedmomentcomptens[\Sxp]{1}{0}-\normalizedmomentcomptens[\Sxm]{1}{0}}\in[0,1]
\end{gather} 
we see that 
\begin{align*}
\normalizedmomentstens[\Sxp]{1} &:= \convexscalar_1 \normalizedmomentstens[\Spp]{1}+(1-\convexscalar_1)\normalizedmomentstens[\Spm]{1}
=\left(\normalizedmomentcomptens[\Sxp]{1}{0}-\normalizedmomentcomptens[\Sxm]{1}{0},0,\normalizedmomentcomptens[\Syp]{0}{1},\normalizedmomentcomptens[\Sym]{0}{1}\right)^T\\
\normalizedmomentstens[\Sxm]{1} &:= \convexscalar_1\normalizedmomentstens[\Smp]{1}+(1-\convexscalar_1)\normalizedmomentstens[\Smm]{1} = \left(0,\normalizedmomentcomptens[\Sxm]{1}{0}-\normalizedmomentcomptens[\Sxp]{1}{0},\normalizedmomentcomptens[\Syp]{0}{1},\normalizedmomentcomptens[\Sym]{0}{1}\right)^T
\intertext{and finally}
\normalizedmomentstens{1} &= \convexscalar_2\normalizedmomentstens[\Sxp]{1}+(1-\convexscalar_2)\normalizedmomentstens[\Sxm]{1}.
\end{align*}
Since $\underline{\normalizedmomentstens[\Sij]{1}}\in\Sij$ we see that $\momentcomptens{0}{0}\left(1,\normalizedmoments[\Sij]\right)^T$ can be realized (with respect to $\mmbasis[1]$) by a distribution function with support in $\Sij$, namely the quarter-moment distribution $\distribution[\Sij]$ (realizing $\momentcomptens{0}{0}\left(1,\underline{\normalizedmomentstens[\Sij]{1}}\right)^T$ with respect to $\pmbasis[\Sij]$) as given in equation \eqref{eq:psiQM} in \lemref{lem:QMFORealizability}.

Therefore, due to the linearity of the problem, a non-negative realizing distribution for $\moments$ is given by
\begin{align}
\label{eq:psiMM}
\distribution = \convexscalar_2\left(\convexscalar_1\distribution[\Spp]+(1-\convexscalar_1)\distribution[\Spm]\right) + (1-\convexscalar_2)\left(\convexscalar_1\distribution[\Smp]+(1-\convexscalar_1)\distribution[\Smm]\right).
\end{align}
\end{proof}

\section{Kershaw closures}
\label{sec:Kershaw}
A typical drawback of the minimum-entropy models defined by \eqref{eq:psiME} is that the dual problem \eqref{eq:dual} cannot be solved analytically. The numerical solution, which has to be calculated at least once in every space-time cell, is challenging and expensive \cite{Hauck2010,Alldredge2014}.

On the other hand, standard $\PN$ models may give physically irrelevant solutions since they do not ensure positivity of the underlying distribution function. 

Due to this, \emph{Kershaw closures} became recently a topic of increasing interest. They are constructed in such a way that they are automatically generated by a nonnegative distribution function $\ansatz$. Therefore, the moment vector including the unknown highest moment is also realizable with respect to the basis of one order higher. Furthermore the flux function $\ints{\SC\basis\ansatz}$ is chosen to be exact (i.e. $\ints{\SC\basis\ansatz} = \ints{\SC\basis\distribution}$ where $\distribution$ realizes the moment vector including the unknown highest moment) if $\moments = \isotropicmoment = \frac{\momentcomptens{0}{0}}{4\pi}\ints{\basis}$ is the isotropic moment. The last condition is also called the \emph{isotropic interpolation condition}. 

In one spatial dimension for a full-moment basis, the Kershaw closure is also exact on the realizability boundary, because there the realizing measure on the realizability boundary is unique. This property is no longer true in higher dimension (see e.g. \cite{Monreal}) or for other models. However, it gives an idea of how to construct such a closure. 

The name of the closure is dedicated to David Kershaw who first proposed such an idea in \cite{Ker76}. For an introduction into Kershaw closures in one space dimension we refer to \cite{Schneider2014}.
The construction of Kershaw closures of first order as done in \cite{Schneider2014} requires realizability information of second order. With this it is possible to linearly interpolate between different parts of the higher-order realizability boundaries (choosing the interpolation parameter in such a way that the isotropic moment is interpolated). The resulting model is then cheap to evaluate because it is analytically closed (in contrast to minimum-entropy models).

Unfortunately, we were not able to provide a closed second-order realizability theory for mixed or quarter moments yet which implies that we can't identify the correct parts of this second-order realizability boundary for interpolation. However, it turns out that under some assumptions on $\normalizedmomentstens[\Sij]{2}$ the second-order realizability information from the half moments in one space dimension are sufficient to define the Kershaw closure for quarter moments in a similar fashion.

For mixed moments we choose a different approach to build the unknown second moment. Abusing the constructive procedure in Theorem \ref{thm:NecSuffFOMM} we are able to provide a closure for mixed moments by combining the quarter-moment closures accordingly.

\subsection{Quarter moments}
It was shown in \cite{Lev84} that, by assuming that the distribution function is symmetric around a preferred direction, the second moment of the $\MN[1]$ closure can be decomposed into 
\begin{align}
\label{eq:FMM1Closure}
\normalizedmomentstens[\sphere]{2} = \cfrac{1-\eddington}{2}I+\cfrac{\left(3\eddington-1\right)\normalizedmomentstens[\sphere]{1}{\normalizedmomentstens[\sphere]{1}}^T}{2\norm{\normalizedmomentstens[\sphere]{1} }{2}^2}
\end{align}
where $\eddington = \eddington\left(\norm{\normalizedmomentstens[\sphere]{1}}{2}\right)$ is the so-called Eddington factor. This implies that $\normalizedmomentstens[\sphere]{1}$ and ${\normalizedmomentstens[\sphere]{1}}^\perp$ are eigenvectors of $\normalizedmomentstens[\sphere]{2}$. 

This is not necessarily true for all $\normalizedmomentstens[\Sij]{1}$ but has to be true on a subset of the quarterspace.
\begin{lem}
\label{lem:QM1Subsets}
Let $\normalizedmomentstens[\Sij]{1}$ fulfil either\\\vskip 0.1cm
{\centering\begin{inparaenum}[(a)]
\item $\normalizedmomentcomptens[\Sij]{1}{0} = 0$,
\item $\normalizedmomentcomptens[\Sij]{0}{1} = 0$,
\item $\abs{\normalizedmomentcomptens[\Sij]{1}{0}} = \abs{\normalizedmomentcomptens[\Sij]{0}{1}}$ or
\item $\norm{\normalizedmomentstens[\Sij]{1}}{2} = 1$.\\\vskip 0.1cm
\end{inparaenum}}
Then for any $\normalizedmomentstens[\Sij]{2}$ which is generated by a realizing distribution for $\normalizedmomentstens[\Sij]{1}$ we have that $\normalizedmomentstens[\Sij]{1}$ is an eigenvector of $\normalizedmomentstens[\Sij]{2}$ with eigenvalue $\eigenvalue$ satisfying $\norm{\normalizedmomentstens[\Sij]{1}}{2}^2 \leq \eigenvalue \leq \norm{\normalizedmomentstens[\Sij]{1}}{2}$.
\end{lem}
\begin{proof}
We only prove the case $\Sij = \Spp$. The other quadrants follow similarly.
Assume that $\normalizedmomentstens[\Spp]{1}$ is generated by $\distribution\geq 0$. Since $\SCx,\SCy\geq 0$ we have the following.
\begin{enumerate}[(a)]
\item If $\normalizedmomentcomptens[\Sij]{1}{0} = 0$ it follows immediately that $\distribution$ has support only on $\SCangle = \frac{\pi}{2}$ and $\SCheight\in[-1,1]$. On this support we have $\SCx = 0$, which implies that
\begin{align}
\label{eq:QM1SecondMomentSupport1}
\normalizedmomentstens[\Spp]{2} &= \begin{pmatrix}
0 & 0\\0 & \normalizedmomentcomptens[\Spp]{0}{2}
\end{pmatrix} \quand\\
\nonumber\normalizedmomentstens[\Spp]{2}\normalizedmomentstens[\Spp]{1} 
&= \begin{pmatrix}
0 & 0\\0 & \normalizedmomentcomptens[\Spp]{0}{2}
\end{pmatrix}
\begin{pmatrix}0 \\\normalizedmomentcomptens[\Spp]{0}{1}\end{pmatrix}
= \normalizedmomentcomptens[\Spp]{0}{2}\normalizedmomentstens[\Spp]{1}
\end{align} 
Similar to the one-dimensional case we observe that (keeping the support of $\distribution$ in mind) 
\begin{align*}
\momentcomptens[\Spp]{0}{2} = \intpp{\SCy^2\distribution} = 2\pi\int\limits_{-1}^1\left(1-\SCheight^2\right)\distribution~d\SCheight \leq  
2\pi\int\limits_{-1}^1\sqrt{1-\SCheight^2}\distribution~d\SCheight = \momentcomptens[\Spp]{0}{1}
\end{align*}
and using Cauchy-Schwarz inequality 
\begin{align*}
\momentcomptens[\Spp]{0}{2}\momentcomptens[\Spp]{0}{0} &= \intpp{\SCy^2\distribution}\intpp{\distribution} = 4\pi^2\int\limits_{-1}^1\left(1-\SCheight^2\right)\distribution~d\SCheight\int\limits_{-1}^1\distribution~d\SCheight \\&\stackrel{\text{CSI}}{\geq}  
4\pi^2\left(\int\limits_{-1}^1\sqrt{1-\SCheight^2}\distribution~d\SCheight\right)^2 = {\momentcomptens[\Spp]{0}{1}}^2.
\end{align*}
Therefore the eigenvalue $\eigenvalue := \normalizedmomentcomptens[\Spp]{0}{2}$ to the eigenvector $\normalizedmomentstens[\Spp]{1}$ has to satisfy
\begin{align*}
\norm{\normalizedmomentstens[\Sij]{1}}{2}^2 = {\momentcomptens[\Spp]{0}{1}}^2 \leq \eigenvalue \leq {\momentcomptens[\Spp]{0}{1}} = \norm{\normalizedmomentstens[\Sij]{1}}{2}.
\end{align*}
\item Analogously to (a)
\item If $\abs{\normalizedmomentcomptens[\Sij]{1}{0}} = \abs{\normalizedmomentcomptens[\Sij]{0}{1}}$ it follows that 
$\distribution$ has support only on $\SCangle = \frac{\pi}{4}$ and $\SCheight\in[-1,1]$. On this support we have $\SCx = \SCy$, which implies that
\begin{align*}
\normalizedmomentstens[\Spp]{2} &= \begin{pmatrix}
\normalizedmomentcomptens[\Spp]{0}{2} & \normalizedmomentcomptens[\Spp]{0}{2}\\\normalizedmomentcomptens[\Spp]{0}{2} & \normalizedmomentcomptens[\Spp]{0}{2}
\end{pmatrix}, \quad\text{leading to}\\
\normalizedmomentstens[\Spp]{2}\normalizedmomentstens[\Spp]{1} 
&= \begin{pmatrix}
\normalizedmomentcomptens[\Spp]{2}{0} & \normalizedmomentcomptens[\Spp]{2}{0}\\\normalizedmomentcomptens[\Spp]{2}{0} & \normalizedmomentcomptens[\Spp]{2}{0}
\end{pmatrix}
\begin{pmatrix}\normalizedmomentcomptens[\Spp]{1}{0} \\\normalizedmomentcomptens[\Spp]{1}{0}\end{pmatrix}
= 2\normalizedmomentcomptens[\Spp]{2}{0}\normalizedmomentstens[\Spp]{1}.
\end{align*} 
In this case we have 
\begin{align*}
\momentcomptens[\Spp]{2}{0} = \intpp{\SCx^2\distribution} = \pi\int\limits_{-1}^1\left(1-\SCheight^2\right)\distribution~d\SCheight \leq  
\pi\int\limits_{-1}^1\sqrt{1-\SCheight^2}\distribution~d\SCheight = \frac{1}{\sqrt{2}}\momentcomptens[\Spp]{1}{0}
\end{align*}
while the lower bound is $\momentcomptens[\Spp]{2}{0}\momentcomptens[\Spp]{0}{0}\geq {\momentcomptens[\Spp]{0}{1}}^2$ as before.
%\begin{align*}
%\momentcomptens[\Spp]{2}{0}\momentcomptens[\Spp]{0}{0} &= \intpp{\SCx^2\distribution}\intpp{\distribution} = 2\pi^2\int\limits_{-1}^1\left(1-\SCheight^2\right)\distribution~d\SCheight\int\limits_{-1}^1\distribution~d\SCheight \\&\leq  
%2\pi^2\left(\int\limits_{-1}^1\sqrt{1-\SCheight^2}\distribution~d\SCheight\right)^2 = {\momentcomptens[\Spp]{0}{1}}^2.
%\end{align*}
Therefore the eigenvalue $\eigenvalue := 2\normalizedmomentcomptens[\Spp]{0}{2}$ to the eigenvector $\normalizedmomentstens[\Spp]{1}$ has to satisfy
\begin{align*}
\norm{\normalizedmomentstens[\Sij]{1}}{2}^2=2{\momentcomptens[\Spp]{1}{0}}^2 \leq \eigenvalue \leq \sqrt{2}{\momentcomptens[\Spp]{1}{0}} = \norm{\normalizedmomentstens[\Sij]{1}}{2}.
\end{align*}

\item This follows exactly in the same way as for full moments, see e.g. \cite{Monreal}. Note that here, $\norm{\normalizedmomentstens[\Sij]{1}}{2}^2 = \eigenvalue = \norm{\normalizedmomentstens[\Sij]{1}}{2} = 1$.
\end{enumerate}

\end{proof}

Numerical tests suggest that for the $\QMN[1]$ model the first normalized moment is indeed always close (but not equal) to an eigenvector of the second normalized moment, i.e. $\winkel\left(\momentstens[\Sij]{1},\eigenvector_1\right)$ remains small (where $\eigenvector_1$ is the eigenvector of $\momentstens[\Sij]{2}$ with smaller enclosing angle between itself and $\momentstens[\Sij]{1}$), see \figref{fig:QM1_eigenvector_deviation}. This has been checked numerically using the $\QMN[1]$-tabulation strategy given in \cite{Frank2006}

\begin{figure} 
\centering
\settikzlabel{fig:QM1_eigenvector_deviationa}
\externaltikz{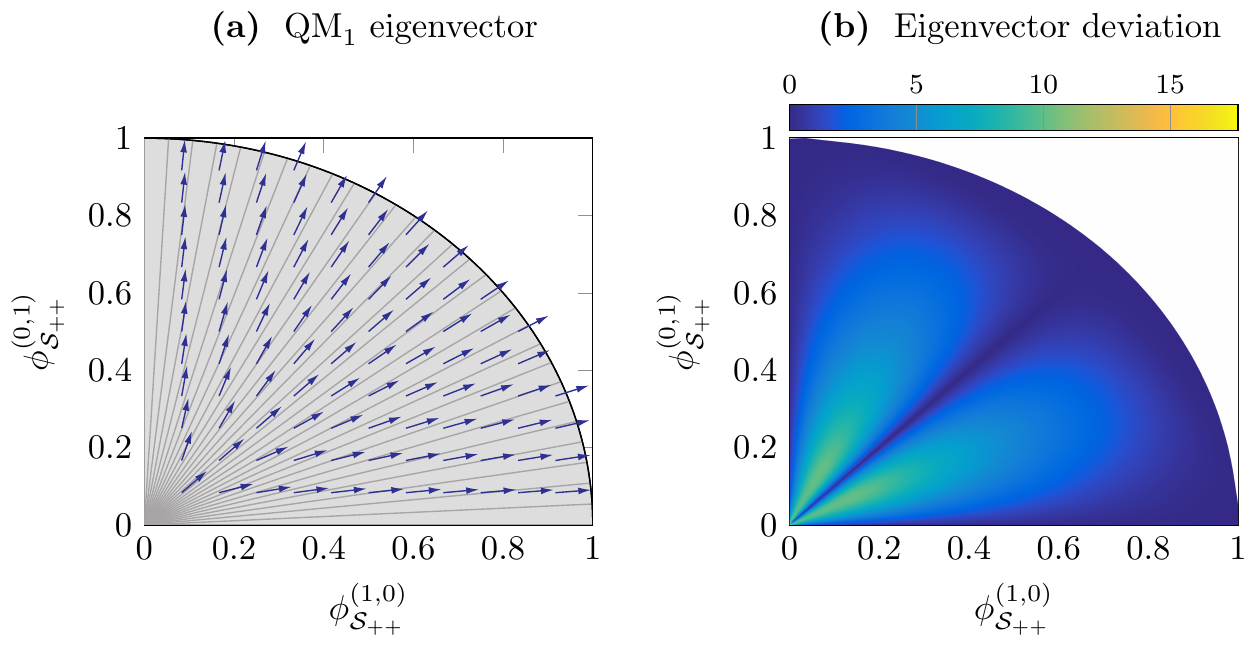}{
	\relinput{Images/QM1_eigenvector_deviation/QM1_eigenvector_deviation.tex}
}
\caption{One eigenvector of the $\QMN[1]$ second moment $\normalizedmomentstens[\Spp]{2}$ and its minimal deviation from $\normalizedmomentstens[\Spp]{1}$, measured in degrees.}
\label{fig:QM1_eigenvector_deviation}
\end{figure}

Since the structure of eigenvectors of the $\QMN[1]$ model is not immediately obvious to us, it seems reasonable to assume the simple form \eqref{eq:FMM1Closure} for $\QKN[1]$ as well:
\begin{align}
\label{eq:QK1closure}
\normalizedmomentstens[\Sij]{2} = \QKalpha{1}I+\QKalpha{2}\cfrac{\normalizedmomentstens[\Sij]{1}{\normalizedmomentstens[\Sij]{1}}^T }{\norm{\normalizedmomentstens[\Sij]{1} }{2}^2}
\end{align}
where the coefficients $\QKalpha{1}(\moments), \QKalpha{2}(\moments)\in\Rpos$ have to be chosen accordingly. The eigenvalue associated to $\normalizedmomentstens[\Sij]{1}$ is given by $\eigenvalue = \QKalpha{1}+\QKalpha{2}$. \\

A short calculation shows that under this assumptions the bounds $\norm{\normalizedmomentstens[\Sij]{1}}{2}^2 \leq \eigenvalue \leq \norm{\normalizedmomentstens[\Sij]{1}}{2}$ have to be fulfilled for all $\normalizedmomentstens[\Sij]{1}\in\Sij$, and not only on the subsets of $\Sij$ shown in \lemref{lem:QM1Subsets}. For every $\unitnormalvector\in\Sij$ it holds that
\begin{align*}
\momentstens[\Sij]{2}\unitnormalvector = \intsij{\SC\underbrace{\left(\unitnormalvector\cdot\SC\right)}_{\leq 1}\distribution}\leq \intsij{\SC\distribution} = \momentstens[\Sij]{1},
\end{align*}
where the last inequality is meant component-wise. This implies for $\unitnormalvector = \frac{\normalizedmomentstens[\Sij]{1}}{\norm{\normalizedmomentstens[\Sij]{1}}{2}}$ that 
\begin{align*}
\momentstens[\Sij]{2}\unitnormalvector = \lambda\unitnormalvector \leq \momentstens[\Sij]{1},
\end{align*}
and consequently $\eigenvalue \leq \norm{\normalizedmomentstens[\Sij]{1}}{2}$.\\
The lower bound on $\eigenvalue$ follows in the same way as for full moments observing that $\momentstens[\Sij]{2}-\momentstens[\Sij]{1}{\momentstens[\Sij]{1}}^T$ has to be positive semi-definite.\\

The construction of the Kershaw closure in one spatial dimension linearly interpolates between the upper and lower boundary values for the second moment in such a way that the isotropic point is correctly hit. We define this interpolation not on the second moment but on its eigenvalues. The isotropic moment vector is given by $\normalizedmomentcomptens[\Spp]{1}{0} = \normalizedmomentcomptens[\Spp]{0}{1} = \frac{1}{2}$ (similarly for the other quadrants with appropriately changed sign), $\normalizedmomentcomptens[\Spp]{2}{0} = \normalizedmomentcomptens[\Spp]{0}{2} = \frac{1}{3}$ and $\normalizedmomentcomptens[\Spp]{1}{1} = \frac{2}{3\pi}$.

We therefore define
\begin{align*}
\eigenvalue = \QKalpha{1}+\QKalpha{2} = \convexscalar\norm{\normalizedmomentstens[\Sij]{1}}{2}^2 + (1-\convexscalar)\norm{\normalizedmomentstens[\Sij]{1}}{2},
\end{align*}
where $\convexscalar = \convexscalar\left(\normalizedmomentstens[\Spp]{1}\right)\in[0,1]$ has to be chosen accordingly. Plugging in the isotropic point we get that $\convexscalar\left(\frac12,\frac12\right) = -\frac{2\, \pi - 3\, \pi\, \sqrt{2} + 4}{3\, \pi\, \left(\sqrt{2} - 1\right)}\approx 0.7801$. 

Although other choices are possible we furthermore assume that $\QKalpha{1}+\QKalpha{2}$ is rotationally symmetric, which implies that $\convexscalar = \convexscalar\left(\norm{\normalizedmomentstens[\Sij]{1}}{2}\right)$. Unfortunately, the simplest choice $\convexscalar \equiv \convexscalar\left(\frac12,\frac12\right)$ leads to a moment system which is no longer hyperbolic. In fact, some of the eigenvalues of the flux Jacobian are imaginary, especially in those regions where the $\QMN[1]$ eigenvectors deviate strongly from $\normalizedmomentstens[\Sij]{1}$ (compare \figref{fig:QM1_eigenvector_deviation}). We fixed this problem by choosing $\convexscalar$ to be linear in $\norm{\normalizedmomentstens[\Sij]{1}}{2}$, i.e.
\begin{align*}
\convexscalar = 1-\sqrt{2}\left(1 - \convexscalar\left(\frac12,\frac12\right)\right)\norm{\normalizedmomentstens[\Sij]{1}}{2}.
\end{align*}

The proof of \lemref{lem:QM1Subsets} reveals more information about the choice of $\QKalpha{1}$. In situations (a), (b) and (d) (i.e. on the realizability boundary), $\QKalpha{1}$ has to be zero, while it has to be $\frac13 -\frac{2}{3\pi}$ at the isotropic point (recall that there $\QKalpha{2} = \frac{2}{3\pi}$ due to the off-diagonals of the second moment). The simplest function fulfilling these two conditions is given by 

\begin{align*}
\QKalpha{1} &= \abs{\normalizedmomentcomptens[\Sij]{1}{0}\, \normalizedmomentcomptens[\Sij]{0}{1}}\, \left(\frac{8}{3}-\frac{16}{3\, \pi}\right)\, \left(1-\norm{\normalizedmomentstens[\Sij]{1}}{2}^2\right)\\
\end{align*}
The corresponding closure relation is shown in \figref{fig:QM1_second_moment}. Despite the fact that the assumption on the eigenvectors for $\QMN[1]$ is not true for all $\normalizedmomentstens[\Sij]{1}$ the resulting second moment for $\QKN[1]$ is very close to the minimum-entropy one. This is shown in \figref{fig:QK1_flux_error}.

\begin{figure} 
	\centering
	\externaltikz{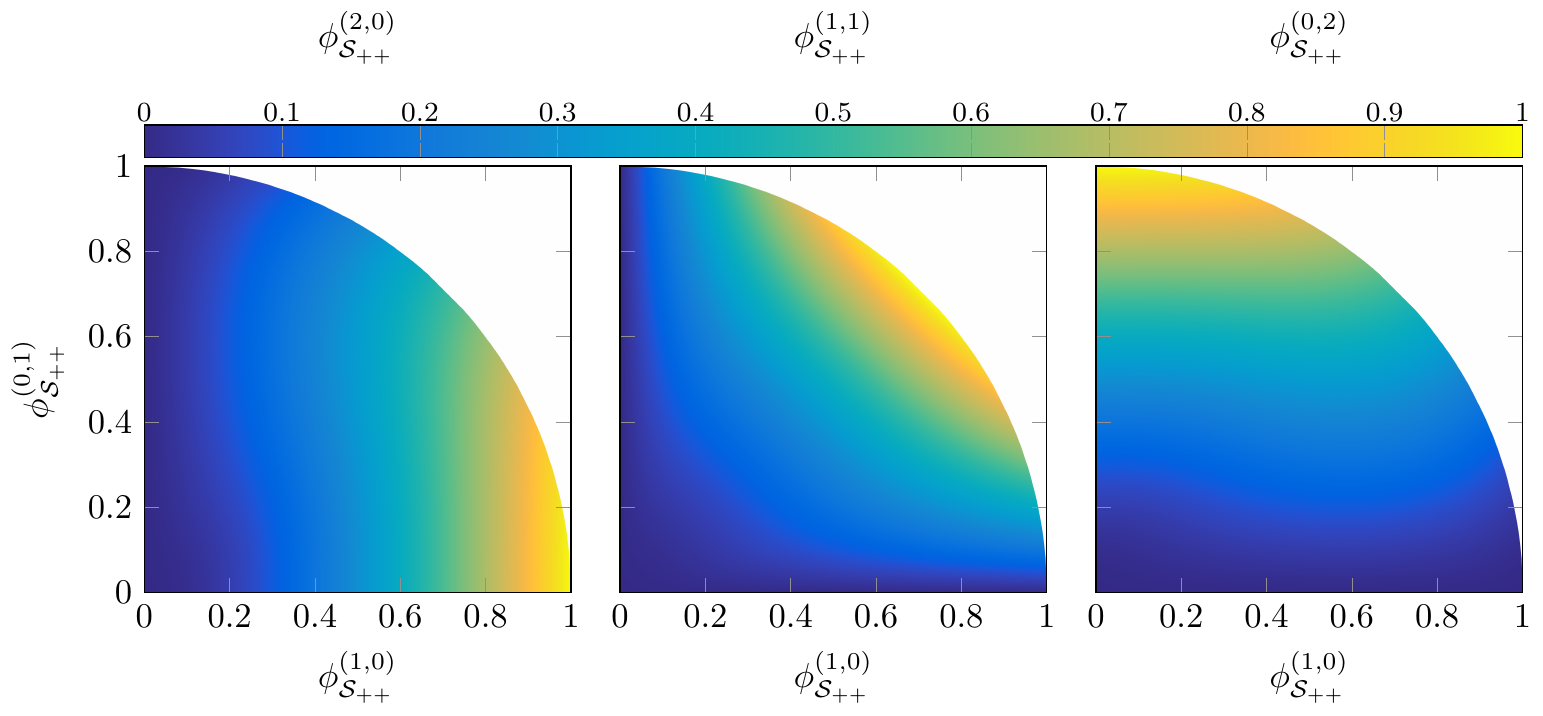}{
		\relinput{Images/QM1_second_moment/QM1_second_moment.tex}
	}
	\caption{Components of the $\QKN[1]$ second moment on $\Spp$.}
	\label{fig:QM1_second_moment}
\end{figure}

\begin{figure} 
	\centering
	\externaltikz{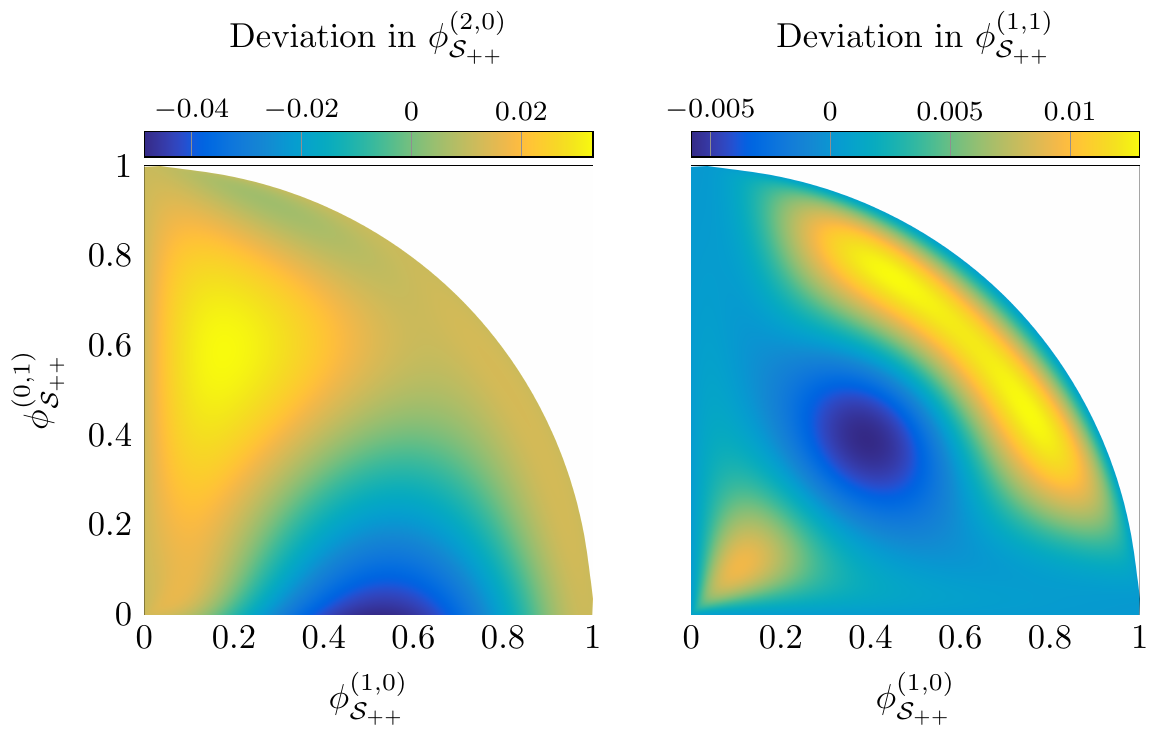}{
		\relinput{Images/QK1_flux_error/QK1_flux_error.tex}
	}
	\caption{Deviation of $\normalizedmomentcomptens[\Spp]{2}{0}$ and $\normalizedmomentcomptens[\Spp]{2}{0}$ for $\QMN[1]$ and $\QKN[1]$, respectively.}
	\label{fig:QK1_flux_error}
\end{figure}

Since we are not able to provide a closed second-order realizability theory we cannot conclude immediately that this closure is in fact realizable with respect to $\pmbasis[2]$. We therefore have to go a different way and, as before, provide a realizing distribution for this special case. We only prove the case $\Sij = \Spp$ and drop the lower index $\Spp$ for readability issues. We make the ansatz\footnote{We acknowledge G. Alldredge for presenting us this ansatz.}
\begin{align}
\distribution = \momentcomptens{0}{0}\sum_{\QKdistind \pm}\QKdistcoeff{\QKdistind\pm} \diracmd(\SC - (\normalizedmomentstens{1} \pm \QKdistdens{\QKdistind\pm}\QKdisteigvector_\QKdistind)) \,,
\label{eq:four-delta-ansatz}
\end{align}
where $\QKdistind \in \{ 1, 2 \}$ and the parameters $\QKdistcoeff{\QKdistind\pm},\QKdistdens{\QKdistind\pm}\in\R$ and $\QKdisteigvector_\QKdistind\in\R^2$ still have to be determined. Without loss of generality we assume $\momentcomptens{0}{0} = 1$ for brevity. The moments of this ansatz are
\begin{align*}
1 &\stackrel{!}{=}  \intpp{\distribution} = \sum_{\QKdistind \pm}\QKdistcoeff{\QKdistind\pm}, \\
\normalizedmomentstens{1} &\stackrel{!}{=} \intpp{\SC \distribution}= \left( \sum_{\QKdistind \pm}\QKdistcoeff{\QKdistind\pm} \right) \normalizedmomentstens{1} + \sum\limits_{\QKdistind} (c_{\QKdistind +} \QKdistdens{\QKdistind +} - \QKdistcoeff{\QKdistind -} \QKdistdens{\QKdistind -})\QKdisteigvector_\QKdistind, \\
\normalizedmomentstens{2} &\stackrel{!}{=} \intpp{\SC \SC^T \distribution} = \left( \sum_{\QKdistind \pm}\QKdistcoeff{\QKdistind\pm} \right) \normalizedmomentstens{1} {\normalizedmomentstens{1}}^T + \sum\limits_{\QKdistind} (\QKdistcoeff{\QKdistind +} \QKdistdens{\QKdistind +} - \QKdistcoeff{\QKdistind -} \QKdistdens{\QKdistind -})\QKdisteigvector_\QKdistind {\normalizedmomentstens{1}}^T \\
  & \hspace{2.5cm}+ \sum\limits_{\QKdistind} \left(\QKdistcoeff{\QKdistind +} \QKdistdens{\QKdistind +}^2 + \QKdistcoeff{\QKdistind -} \QKdistdens{\QKdistind -}^2 \right)\QKdisteigvector_\QKdistind \QKdisteigvector_\QKdistind^T.
\end{align*}
Hierarchically inserting the equations into the one of higher order we see that the coefficients $\QKdistcoeff{\QKdistind\pm}, \QKdistdens{\QKdistind \pm}$ fulfil the system of equations
\begin{subequations}
\label{eq:QK1equations}
\begin{align}
 1 &= \sum_{\QKdistind \pm}\QKdistcoeff{\QKdistind\pm}& \label{eq:QK1equationsa}\\
 0 &= \QKdistcoeff{\QKdistind +} \QKdistdens{\QKdistind +} - \QKdistcoeff{\QKdistind -} \QKdistdens{\QKdistind -} \,, & \QKdistind \in \{1, 2\}, \label{eq:QK1equationsb}\\
 \QKdisteig{\QKdistind} &= \QKdistcoeff{\QKdistind +} \QKdistdens{\QKdistind +}^2 + \QKdistcoeff{\QKdistind -} \QKdistdens{\QKdistind -}^2 \,, & \QKdistind \in \{1, 2\}.\label{eq:QK1equationsc}
\end{align}
\end{subequations}
and $(\QKdisteig{\QKdistind},\QKdisteigvector_\QKdistind)$ is the $\QKdistind$-th eigenvalue-eigenvector pair of $\normalizedmomentstens{2}-\normalizedmomentstens{1} {\normalizedmomentstens{1}}^T$, i.e.
\begin{align*}
\QKdisteigvector_1 &= \cfrac{{\normalizedmomentstens{1}}^\perp}{\norm{\normalizedmomentstens{1}}{2}} & \QKdisteig{1} &= \QKalpha{1}\\
\QKdisteigvector_2 &= \cfrac{\normalizedmomentstens{1}}{\norm{\normalizedmomentstens{1}}{2}} & \QKdisteig{2} &= \QKalpha{1}+\QKalpha{2}-\norm{\normalizedmomentstens{1}}{2}^2
\end{align*}
Since $\distribution$ has to be supported in $\Spp$ and should be nonnegative we also get the inequality-constraints
\begin{subequations}
\label{eq:QK1inequalities}
\begin{align}
\QKdistcoeff{\QKdistind\pm}\geq 0 \label{eq:cgeq}\\
\normalizedmomentstens{1} \pm \QKdistdens{\QKdistind\pm}\QKdisteigvector_\QKdistind \in \Spp \label{eq:QSconstraints}
\end{align}
\end{subequations}
Combining \eqref{eq:cgeq} and \eqref{eq:QK1equationsb} shows that $\sign{\QKdistdens{\QKdistind+}} = \sign{\QKdistdens{\QKdistind-}}$, thus, without loss of generality, the quarter-sphere constraints \eqref{eq:QSconstraints} can be rewritten as
\begin{align*}
0\leq \QKdistdens{1+} &\leq \min\left(\cfrac{\normalizedmomentcomptens[\Spp]{1}{0}}{\normalizedmomentcomptens[\Spp]{0}{1}}\norm{\normalizedmomentstens{1}}{2},\sqrt{1-\norm{\normalizedmomentstens{1}}{2}^2}\right) =: \QKdistdensub{1+}\\
0\leq \QKdistdens{1-} &\leq \min\left(\cfrac{\normalizedmomentcomptens[\Spp]{0}{1}}{\normalizedmomentcomptens[\Spp]{1}{0}}\norm{\normalizedmomentstens{1}}{2},\sqrt{1-\norm{\normalizedmomentstens{1}}{2}^2}\right)=: \QKdistdensub{1-}\\
0\leq \QKdistdens{2+} &\leq 1-\norm{\normalizedmomentstens{1}}{2}=: \QKdistdensub{2+}\\
0\leq \QKdistdens{2-} &\leq \norm{\normalizedmomentstens{1}}{2} =: \QKdistdensub{2-}
\end{align*}
A visualization of these bounds can be found in \figref{fig:QK1coefficientsvis} for a specific example. The $\min$ function ensures that the first intersection with either the coordinate axes ($\QKdistdens{1+}$ in the example) or the norm-bound ($\QKdistdens{1-}$ in the example) is taken. For $\QKdistdens{2\pm}$ there is no distinction necessary since we can only intersect with one of them.

\begin{figure}
\centering
\externaltikz{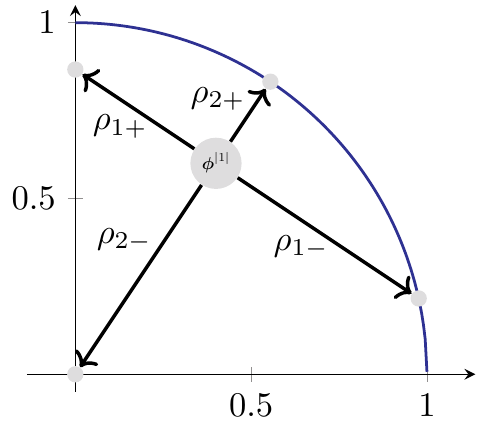}{\input{Images/QK1ansatz.tex}}
\caption{Visualization of the coefficients $\QKdistdens{\QKdistind\pm}$ for the example $\normalizedmomentstens[\Spp]{1}=\left(0.4,0.6\right)^T$. The length of the arrows represent the maximal value for the respective coefficient.}
\label{fig:QK1coefficientsvis}
\end{figure}
Solving \eqref{eq:QK1equationsb} and \eqref{eq:QK1equationsc} for $\QKdistcoeff{\QKdistind\pm}$ we get
\begin{align*}
\QKdistcoeff{\QKdistind\pm} = \cfrac{\QKdisteig{\QKdistind}}{\QKdistdens{\QKdistind\pm}\left(\QKdistdens{\QKdistind+}+\QKdistdens{\QKdistind-}\right)}\;.
\end{align*}
Plugging this into \eqref{eq:QK1equationsa} we obtain an equation for $\QKdistdens{\QKdistind\pm}$:
\begin{align}
\QKdisteig{2}\QKdistdens{1+}\QKdistdens{1-}+\QKdisteig{1}\QKdistdens{2+}\QKdistdens{2-} = \QKdistdens{1+}\QKdistdens{1-}\QKdistdens{2+}\QKdistdens{2-}  \label{eq:QK1equationsa_mod}
\end{align}
We found that 
\begin{align*}
\QKdistdens{1-} &= \min\left(\cfrac{\normalizedmomentcomptens[\Spp]{0}{1}}{\normalizedmomentcomptens[\Spp]{1}{0}}\norm{\normalizedmomentstens{1}}{2},\sqrt{1-\norm{\normalizedmomentstens{1}}{2}^2}\right)\\
\QKdistdens{2-} &= \norm{\normalizedmomentstens{1}}{2}\\
\QKdistdens{\QKdistind+} &= \cfrac{2\QKdisteig{\QKdistind}}{\QKdistdens{\QKdistind-}},\qquad \QKdistind \in \{1, 2\}
\end{align*}
solves \eqref{eq:QK1equationsa_mod} while fulfilling \eqref{eq:QSconstraints}. 

We want to prove this exemplarily for $\QKdistdens{2+}$. After some symbolic simplifications we see that
\begin{align*}
\QKdistdens{2+} &= \frac{2\, \sqrt{2}\, \left(4-\pi\right)\, \left(\sqrt{2} + 1\right)}{3\, \pi}\left(\norm{\normalizedmomentstens{1}}{2}-\norm{\normalizedmomentstens{1}}{2}^2\right)\geq 0.
\end{align*}
With this it immediately follows that $1-\norm{\normalizedmomentstens{1}}{2}-\QKdistdens{2+}\geq 0$.

Note that the above choice of the coefficients $\QKdistdens{\QKdistind\pm}$ is not unique. Since the corresponding values for $\QKdistcoeff{\QKdistind\pm}$ are also non-negative, \eqref{eq:four-delta-ansatz} is a non-negative distribution function with support in $\Spp$ which realizes the desired moments. Therefore, the $\QKN[1]$ closure is realizable. 

\subsection{Mixed moments}
Since the mixed-moment setting is even more complicated than the quarter-moment one, we want to follow a different approach to construct an approximate closure. As has been shown in Theorem \ref{thm:NecSuffFOMM} it suffices to have a non-negative quarter-moment distribution $\distribution[\Sij]$ to construct a mixed-moment reproducing distribution. If these distributions fulfil the isotropic moment interpolation condition for the quarter-moment basis, by linearity, the resulting mixed-moment distribution \eqref{eq:psiMM} will fulfil it as well for the mixed-moment basis.

Such a distribution function for quarter moments has been found in \eqref{eq:four-delta-ansatz}. Thus setting $\underline{\normalizedmomentstens[\Sij]{1}}$ as in \eqref{eq:phiij} will provide the desired mixed-moment distribution function $\distribution[{\mmbasis[1]}]$ using the interpolation defined by \eqref{eq:MM1Interpolation}.

A closure can now be generated by calculating the second moments with respect to those quarter-space distributions which have support in the corresponding set $\spheresubset$. For example, 
\begin{align*}
\momentcomptens[\Sxp]{2}{0} = \intsxp{\SCx^2\distribution[{\mmbasis[1]}]} = \intpp{\SCx^2\distribution[{\mmbasis[1]}]}+\intpm{\SCx^2\distribution[{\mmbasis[1]}]} =  \convexscalar_{2}\convexscalar_{1}\momentcomptens[\Spp]{2}{0} + (1-\convexscalar_{2})\convexscalar_{1}\momentcomptens[\Spm]{2}{0}
\end{align*}
where $\momentstens[\Sij]{2}$ is defined as in \eqref{eq:QK1closure} with ${\normalizedmomentstens[\Sij]{1}}$ substituted by $\underline{\normalizedmomentstens[\Sij]{1}}$.

\begin{rem}
Numerical experiments suggest that this closure is also (weakly) hyperbolic. No imaginary eigenvalues of the flux Jacobian have been found on a finely discretized grid of the realizable set. However, at the isotropic point both flux Jacobians have three times the eigenvalue $0$. Still, the transformation matrix has full rank, i.e. the geometric multiplicity of the eigenvalue $0$ is $3$.
\end{rem}
\begin{rem}
Any other choice of $\distribution[\Sij]$ which fulfils the desired properties with respect to the quarter-moment basis will be a feasible choice as well. Exemplarily, the quarter-moment minimum-entropy closure $\QMN[1]$ shall be mentioned. An efficient implementation using tabulation is given in \cite{Frank2006}.
\end{rem}

Note that this closure is not an approximation of the $\MMN[1]$ model and might behave differently. This is shown exemplarily in \figref{fig:MM1MQ1_flux_Difference} where the second moment $\normalizedmomentcomptens[\Syp]{0}{2}$ for the $\MMN[1]$ model and the Kershaw closure is plotted against $\normalizedmomentcomptens[\Sxp]{1}{0}$ and $\normalizedmomentcomptens[\Sxm]{1}{0}$ with $\normalizedmomentcomptens[\Syp]{0}{1} = -\normalizedmomentcomptens[\Syp]{0}{1} = \frac14$ fixed. The corresponding shape of the Kershaw closure is very different to those of the $\MMN[1]$ model.

\begin{figure} 
	\centering
	\externaltikz{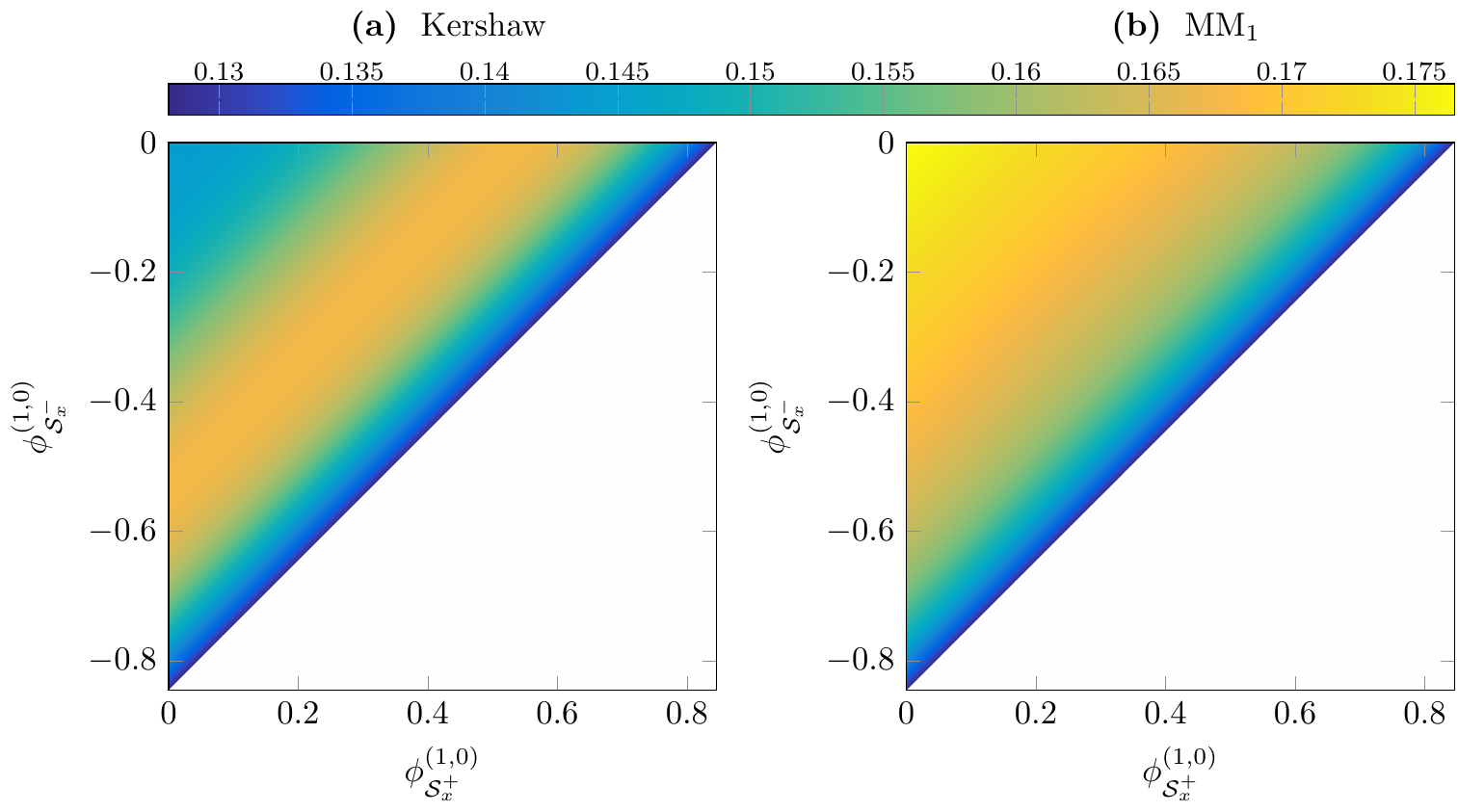}{
		\relinput{Images/tp_MM1MQ1_flux_Difference/tp_MM1MQ1_flux_Difference.tex}
	}
	\caption{Second moment $\normalizedmomentcomptens[\Syp]{0}{2}$ for the $\MMN[1]$ and Kershaw model for $\normalizedmomentcomptens[\Syp]{0}{1} = -\normalizedmomentcomptens[\Sym]{0}{1} = \frac14$.}
	\label{fig:MM1MQ1_flux_Difference}
\end{figure}

\subsection{Treatment of the Laplace-Beltrami operator}
\label{sec:TreatmentLB}
A clear drawback of the mixed-moment Kershaw closure is that the moments of the Laplace-Beltrami operator (see \secref{sec:MixedMomentsLB}) require the evaluation of the ansatz function. When using the $\QKN[1]$ distribution to realize the mixed-moment closure, the ansatz is a linear combination of Dirac deltas, which cannot be evaluated in a feasible way.\\

One could follow the authors in \cite{Frank07} and replace the ansatz in the Laplace-Beltrami moments by the corresponding polynomial distribution \eqref{eq:PnAnsatz}. Then, we have $\ints{\mmbasis[1]\LaplaceBeltrami\distribution}\approx \MKLBmatrix\moments$ with
\begin{equation}
\label{eq:MK1polyLB}
\MKLBmatrix = \left(\begin{array}{ccccc} 0 & 0 & 0 & 0 & 0\\ \MKLBconst{0} & \MKLBconst{3} & \MKLBconst{2} & \MKLBconst{1} & - \MKLBconst{1}\\ - \MKLBconst{0} & \MKLBconst{2} & \MKLBconst{3} & - \MKLBconst{1} & \MKLBconst{1}\\ \MKLBconst{0} & \MKLBconst{1} & - \MKLBconst{1} & \MKLBconst{3} & \MKLBconst{2}\\ - \MKLBconst{0} & - \MKLBconst{1} & \MKLBconst{1} & \MKLBconst{2} & \MKLBconst{3} \end{array}\right),
\qquad
\begin{aligned}[c]
\MKLBconst{0} &= - \frac{3}{\pi - 4} - 1\\
\MKLBconst{1} &= \frac{3\, \pi\, \left(\pi - 3\right)}{\left(2\, \pi - 4\right)\, \left(\pi - 4\right)}\\
\MKLBconst{2} &=  - \frac{3\, \pi}{\left(2\, \pi - 4\right)\, \left(\pi - 4\right)} - \frac{3}{2}\\
\MKLBconst{3} &=\frac{3\, \pi}{\left(2\, \pi - 4\right)\, \left(\pi - 4\right)} - \frac{1}{2}.
\end{aligned}
\end{equation}
Unfortunately, this closure may lead to an undesired behaviour of the moment-system \eqref{eq:MomentSystemClosed}, since its solution may leave the realizable set (this can be easily shown even in one spatial dimension with the ansatz given in \cite{Frank07}). This follows from the fact that the tangent field spanned by $\MKLBmatrix\moments$ may point outside of the realizable set at those parts of the boundary where the polynomial ansatz \eqref{eq:PnAnsatz} is negative.\\

Consider for example $\moments = \left(1,1,0,0,0\right)^T$. Then we have that 
\begin{align*}
\MKLBmatrix\moments
 &\approx \left( 0, -2.813, 0.813, 1.813, -1.813\right)^T,
\end{align*}
pointing outside the realizable set in the third component (recall $\momentcomptens[\Sxm]{1}{0}\leq 0$).\\

We therefore modify this approach slightly and replace the Dirac ansatz with the tabulated $\QMN[1]$ distribution. Since the fluxes for $\QMN[1]$ and $\QKN[1]$ are very similar, it can be expected that the error done with this replacement is negligible. Furthermore, since the ansatz is positive, the moment-system \eqref{eq:MomentSystemClosed} will behave as expected. Indeed, calculating the approximation for the previous example we get approximately
\begin{align*}
\ints{\mmbasis[1]\LaplaceBeltrami\ansatz[\moments]} 
 &\approx \left( 0, -2, 0, \MKLBconst{}, -\MKLBconst{}\right)^T,~\MKLBconst{}\to\infty,
\end{align*}
which points inside of the realizable set. Note the fourth and fifth component which converges to $\pm\infty$. This is consistent with the original problem since on the realizability boundary, the exponential ansatz converges to a combination of Dirac deltas.
\def\tikzpath{Images/}
\section{Numerical results}
\label{sec:results}
We present the derived models in several benchmark test-cases. The numerical discretizations are obtained with a two-dimensional generalization of the high-order realizability-preserving discontinuous-Galerkin scheme given in \cite{Schneider2015a} for the Kershaw models and a generalization of the realizability-preserving kinetic scheme given in \cite{Schneider2015b} for minimum-entropy models. Those generalizations are in principle straight-forward but might be topic of a follow-up paper. The spatial and temporal order is four where roughly $10000$ rectangular (discontinuous-Galerkin scheme) and $20000$ triangular elements (kinetic scheme) in space were used. The reference solution is discretized with the second-order scheme given in \cite{Roth2014,Carrillo2014,roth2014numerical}.
\subsection{Line Source}
The line-source test is a Green's function problem, where a pulse of particles is emitted in an infinite medium \cite{Brunner2005}. It has been widely investigated for isotropic scattering in \cite{Garrett2014}. We choose the following set of parameters for this problem, smoothing the initial Dirac delta in space:
\begin{itemize}
\item Domain: $\Domain = [-\frac12,\frac12]\times[-\frac12,\frac12]$
\item Final time $\tf=0.45$
\item Parameters: $\absorption=\source = 0$, $\scattering = 1$
\item Initial condition: $\distributiontzero(\spatialVariable,\SC) = \max(\frac{\exp\left({-\frac{\x^2 + \y^2}{2\, {\LinesourceWidth}^2}}\right)}{8\, \pi\, {\LinesourceWidth}^2},10^{-4})$ with $\LinesourceWidth=0.03$
\item Boundary conditions: Isotropic in $\SC$, consistent with initial conditions.
\end{itemize}

We show several model solutions in Figures \ref{fig:Linesource2DCuts} and \ref{fig:Linesource1DCuts}. We observe that convergence is achieved quickly with increasing moment order $\momentorder$ for $\MN$ models which can be observed as well in one spatial dimension. This is a huge contrast to the integral-scattering operator used in \cite{Garrett2014}, where only slow convergence can be observed.\\ 
Note that, by construction, the $\MMN[1]$ model has no rotational symmetry. The main directions of propagation follow the half-spaces, i.e. along the cartesian axes. Still, it performs significantly better than $\MN[1]$ (especially along the diagonal cut shown in \figref{fig:Linesource1DCutsdiagonal}), but worse than $\MN[2]$ which almost converged to the reference solution. We observed that the naive implementation of higher-order mixed-moment closures like in the $\MMN[2]$ model are numerically instable due to errors in the non-linear closure of the Laplace-Beltrami operator (see \secref{sec:MixedMomentsLB}).

Figures \ref{fig:Linesource2DCutsMK1poly} and \ref{fig:Linesource2DCutsMK1} show the $\MKN[1]$ model with the two different closures for the Laplace-Beltrami operator given in \secref{sec:TreatmentLB}. The polynomial closure \eqref{eq:MK1polyLB} has a slightly lower peak speed as the tabulated closure but more interesting is that the latter is much closer to the $\MMN[1]$ solution (\figref{fig:Linesource2DCutsMM1}). This shows, that in contrast to the one-dimensional situation in \cite{Frank07}, the choice of the Laplace-Beltrami closure has a significant effect on the solution of the mixed-moment system.

Since $\scattering>0$, the $\QKN[1]$ model cannot be applied here.

\begin{figure}[htbp]
\settikzlabel{fig:Linesource2DCutsRef}\settikzlabel{fig:Linesource2DCutsM1}\settikzlabel{fig:Linesource2DCutsM2}
\settikzlabel{fig:Linesource2DCutsMM1}\settikzlabel{fig:Linesource2DCutsMK1poly}\settikzlabel{fig:Linesource2DCutsMK1}
\centering
\externaltikz{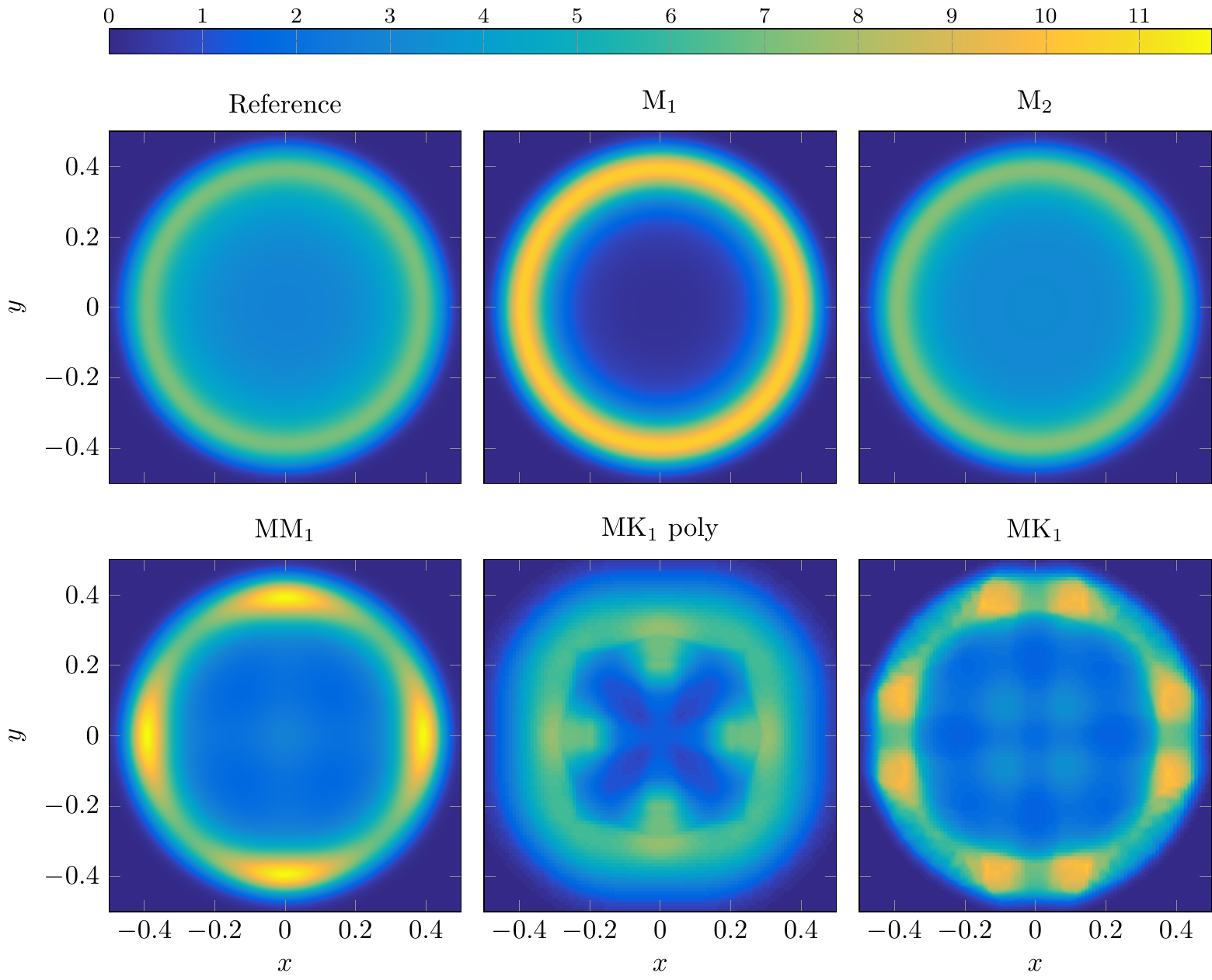}{\input{\tikzpath Linesource2DCuts}}
\caption{Local particle density $\momentcomptens{0}{0}$ of some models at $\timevar = 0.45$ in the line-source test case.}
\label{fig:Linesource2DCuts}
\end{figure}

\begin{figure}[htbp]
\settikzlabel{fig:Linesource1DCutshorizontal}\settikzlabel{fig:Linesource1DCutsdiagonal}
\centering
\externaltikz{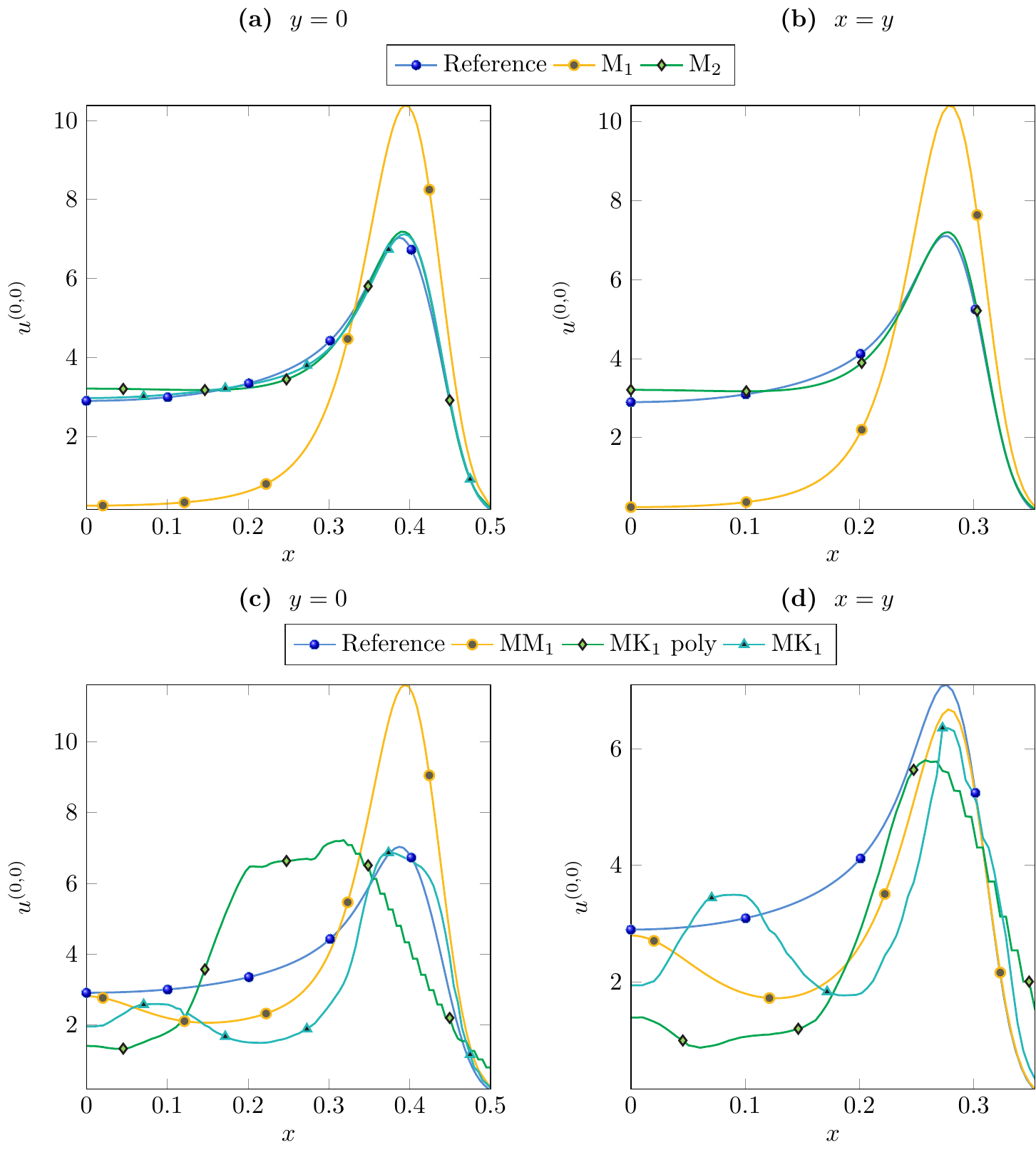}{\input{\tikzpath Linesource1DCuts}}
\caption{Local particle density $\momentcomptens{0}{0}$ of some models at $\timevar = 0.45$ in the line-source test case, horizontal and diagonal cuts.}
\label{fig:Linesource1DCuts}
\end{figure}

\subsection{Two Beams}
This is the typical situation where the $\MN[1]$ model completely fails and is an often-used benchmark problem in 1D (see e.g. \cite{Frank07,Schneider2014}). Instead of using opposing beams we investigate two orthogonal beams hitting each other in a void. We show two different variants of this. The first one is given by

\begin{figure}[h!]
\settikzlabel{fig:TwoBeams2DCutsRef}\settikzlabel{fig:TwoBeams2DCutsM1}\settikzlabel{fig:TwoBeams2DCutsM2}\settikzlabel{fig:TwoBeams2DCutsM3}
\settikzlabel{fig:TwoBeams2DCutsMM1}\settikzlabel{fig:TwoBeams2DCutsMM2}\settikzlabel{fig:TwoBeams2DCutsMM3}\settikzlabel{fig:TwoBeams2DCutsMK1}\settikzlabel{fig:TwoBeams2DCutsQK1}
\centering
\externaltikz{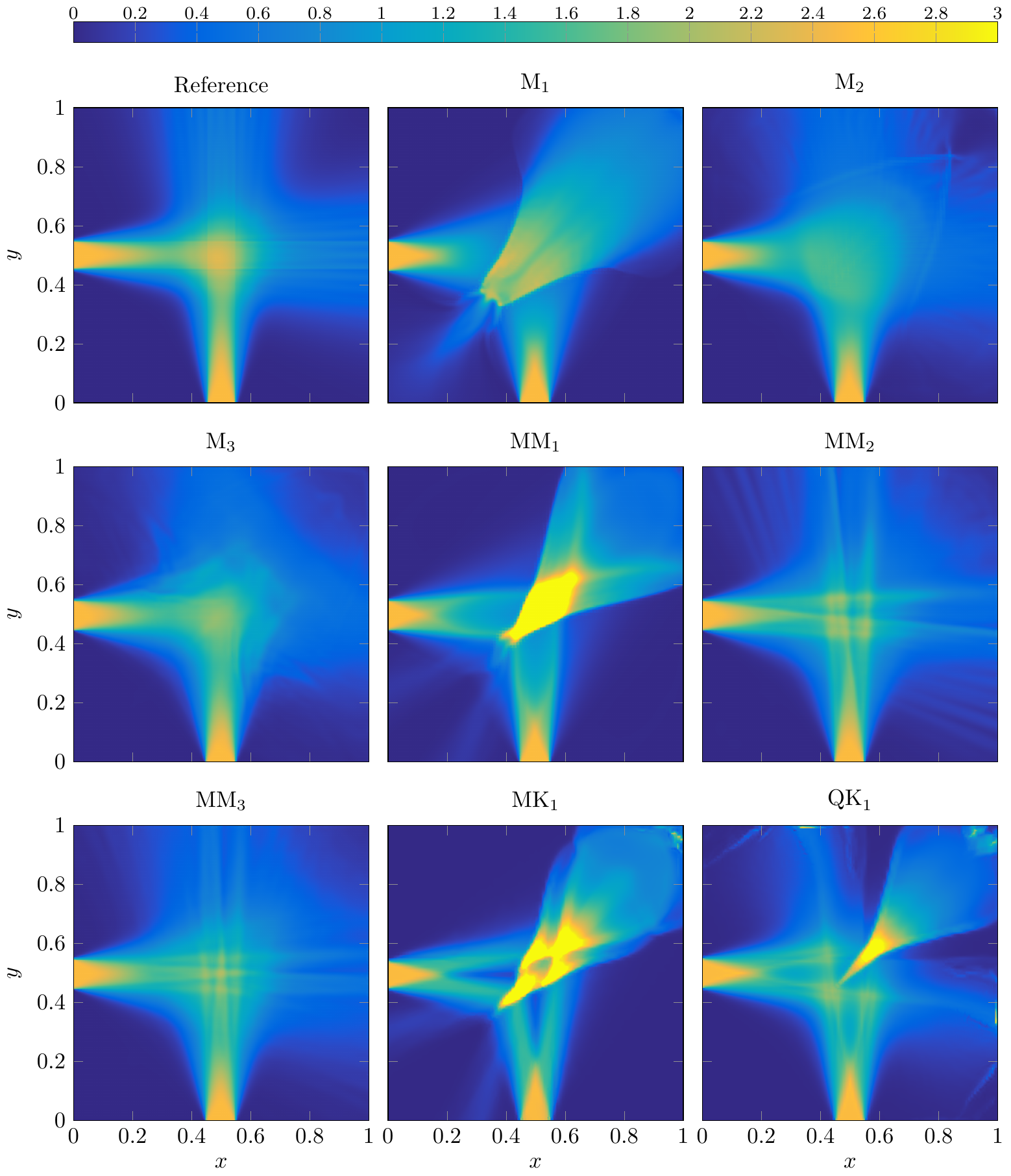}{\input{\tikzpath TwoBeams2DCuts2}}
\caption{Local particle density $\momentcomptens{0}{0}$ of some models at $\timevar = 1.2$ in the two-beams test case. Colorscale cut at $\momentcomptens{0}{0} = 3$.}
\label{fig:TwoBeams2DCuts2}
\end{figure}
\begin{figure}[htbp]
%\settikzlabel{fig:NumberOfFacetsMn}\settikzlabel{fig:NumberOfFacetsMMn}
\centering
\externaltikz{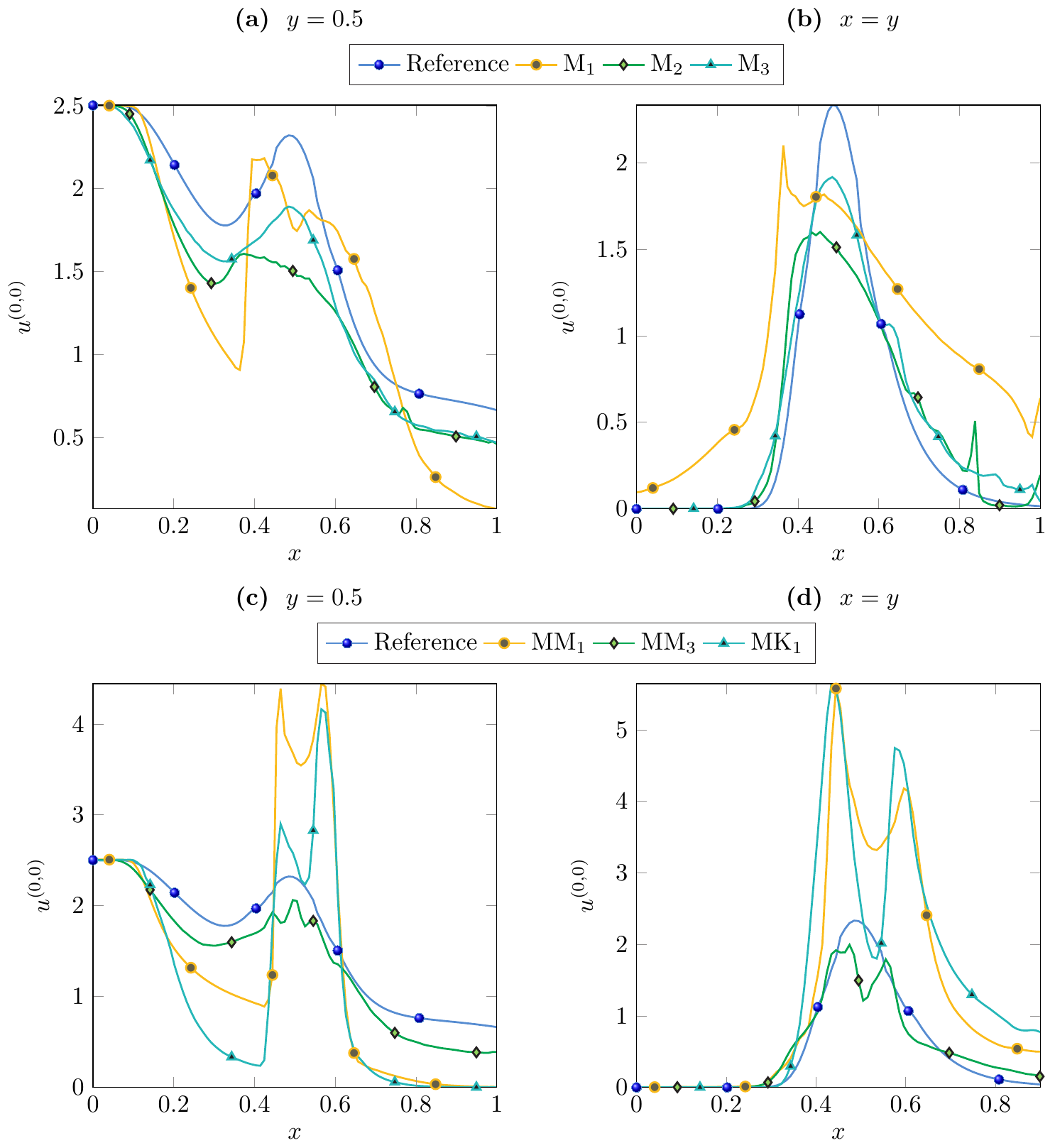}{\input{\tikzpath TwoBeams1DCuts1}}
\caption{Local particle density $\momentcomptens{0}{0}$ of some models at $\timevar = 1.2$ in the two-beams test case, horizontal and diagonal cuts.}
\label{fig:TwoBeams1DCuts}
\end{figure}

\begin{itemize}
\item Domain: $\Domain = [0,1]\times[0,1]$
\item Final time: $\tf = 1.2$
\item Parameters: $\scattering=\absorption=\source=0$
\item Initial condition: $\distributiontzero(\spatialVariable,\SC) = \cfrac{10^{-4}}{4\pi}$
\item Boundary conditions: 
\begin{align*}
\distributionboundary(\timevar,\spatialVariable,\SC) = \cfrac{100}{4\pi}
\begin{cases}
\exp\left(-\cfrac{\SCheight^2+\SCangle^2}{2\TwoBeamsWidth^2}\right) & \text{ if } \x=0,~\y~\in[0.45,0.55]\\
\exp\left(-\cfrac{\SCheight^2+\left(\SCangle-\frac{\pi}{2}\right)^2}{2\TwoBeamsWidth^2}\right) & \text{ if } \y=0,~\x~\in[0.45,0.55]\\
10^{-6} & \text{ else }
\end{cases}
\end{align*}
where $\TwoBeamsWidth^2 = 0.05$.
\end{itemize}
\figref{fig:TwoBeams2DCuts2} shows the solutions of full- and mixed-moment minimum-entropy models up to order $\momentorder=3$, 

\begin{figure}[htbp]
\centering
\externaltikz{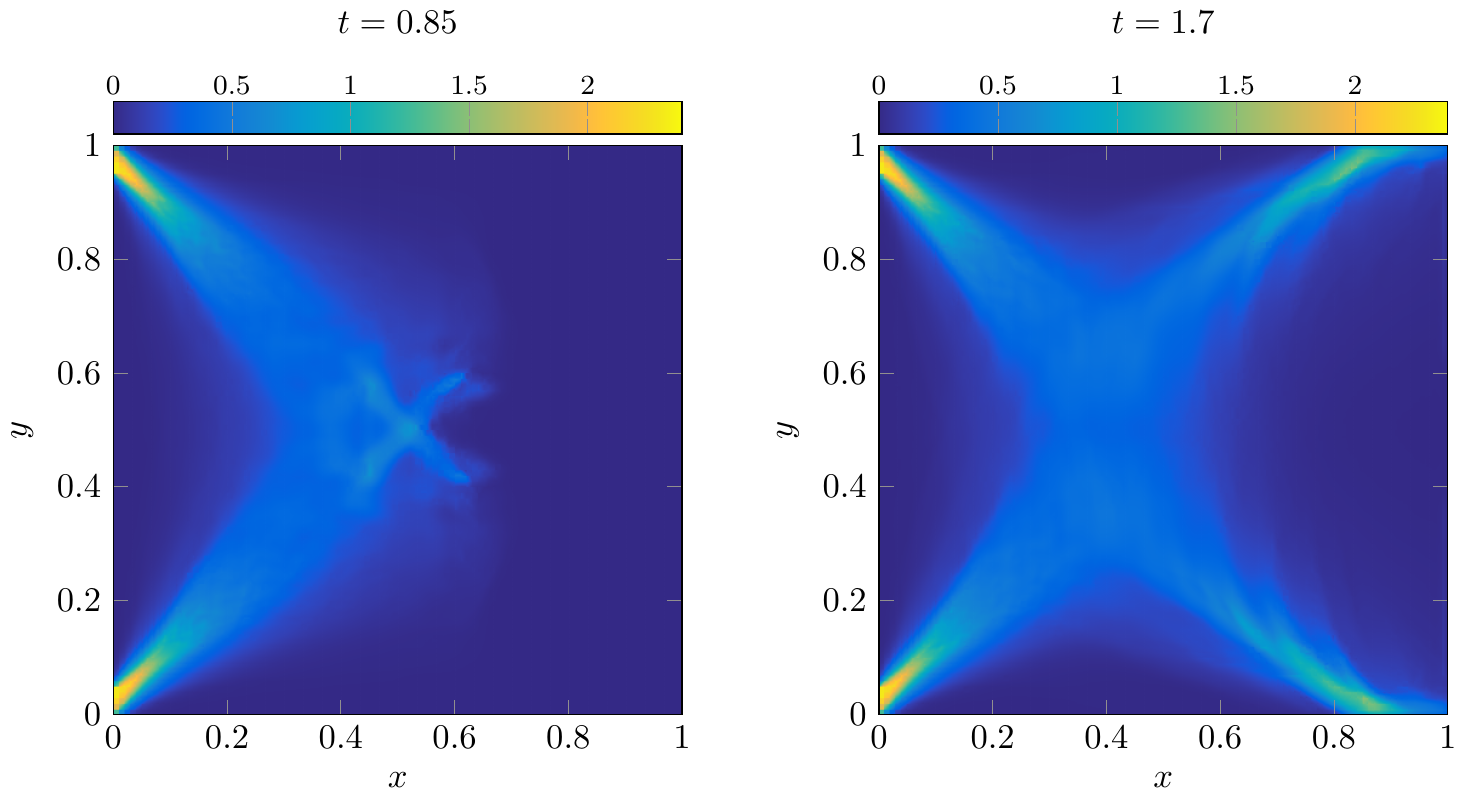}{\input{\tikzpath TwoBeamsRot2DCuts}}
\caption{Local particle density $\momentcomptens{0}{0}$ of the $\MKN[1]$ model at different times in the rotated two-beams test case.}
\label{fig:TwoBeamsRot2DCuts}
\end{figure}

In the second version, the whole setup is rotated clockwise. The two beams sit in the upper-left and lower-left corner, respectively, both pointing to the center of the domain. Due to the rotational symmetry of the full-moment models and the reference solution, only the $\MKN[1]$ closure is shown in \figref{fig:TwoBeamsRot2DCuts}. It is visible that in contrast to the first version of this problem (compare \figref{fig:TwoBeams2DCutsMK1}) no shock is produced when the beams hit.

\section{Conclusions and future work}
\label{sec:conclusions}
We developed a two-dimensional extension of mixed moments, leading to a generalization of the original minimum-entropy $\MMN$ model, proposed in \cite{Frank07,Schneider2014}. Additionally, we provided a first-order realizability theory for mixed as well as quarter moments. 
Since the numerical inversion of the general minimum-entropy moment problem is very expensive we developed approximate, but still realizable, closures, the so-called quarter- and mixed-moment Kershaw closures. They have approximately the same cost as the corresponding polynomial closure while being realizable and non-linear, better adapting to the correct eigenvalues of the true Fokker-Planck solution. 
Although first-order mixed moments completely resolved the zero-netflux problem of the $\MN[1]$ model, a first order approximation in two dimensions is in general not enough, as has been shown in the two-beams test-case. Under some assumptions on the alignment of the beams this drawback can be completely removed. On the other hand, going to higher-order approximations, like $\MN$ or $\MMN$ with $\momentorder\geq 2$, a much better approximation can be done. 
Thus it seems reasonable to extend the concept of Kershaw closures for either $\KN$ or $\MKN$ for $\momentorder\geq 2$ to obtain a cheap and accurate approximation of the Fokker-Planck equation. Furthermore, more research is necessary to obtain better-quality high-order numerical solutions of Kershaw closures.
\appendix
\section{Calculation of partial-moment integrals}
\label{sec:CalculationPMintegrals}
\begin{align*}
\intpp{\SCx^{\basisindx}\SCy^{\basisindy}} &= \frac{1}{2}\betafun\left(\frac{1}{2},1+\frac{1}{2}\left(\basisindx+\basisindy\right)\right)\betafun\left(\frac{\basisindx+1}{2},\frac{\basisindy+1}{2}\right)\\
\intmp{\SCx^{\basisindx}\SCy^{\basisindy}} &= \frac{\left(-1\right)^{\basisindx}}{2}\betafun\left(\frac{1}{2},1+\frac{1}{2}\left(\basisindx+\basisindy\right)\right)\betafun\left(\frac{\basisindx+1}{2},\frac{\basisindy+1}{2}\right)\\
\intmm{\SCx^{\basisindx}\SCy^{\basisindy}} &= \frac{\left(-1\right)^{\basisindx+\basisindy}}{2}\betafun\left(\frac{1}{2},1+\frac{1}{2}\left(\basisindx+\basisindy\right)\right)\betafun\left(\frac{\basisindx+1}{2},\frac{\basisindy+1}{2}\right)\\
\intpm{\SCx^{\basisindx}\SCy^{\basisindy}} &= \frac{\left(-1\right)^{\basisindy}}{2}\betafun\left(\frac{1}{2},1+\frac{1}{2}\left(\basisindx+\basisindy\right)\right)\betafun\left(\frac{\basisindx+1}{2},\frac{\basisindy+1}{2}\right)\\
\intsxp{\SCx^\basisind} &= (-1)^\basisind\intsxm{\SCx^\basisind} = \intsyp{\SCy^\basisind} = (-1)^\basisind\intsym{\SCy^\basisind} = \frac{2\pi}{\basisind+1}
\end{align*}
where $\betafun\left(z,w\right)$ is the beta function.

% Bibliography
%%%%%%%%%%%%%%
\bibliographystyle{siam}
\bibliography{bibliography}

\end{document}